\documentclass[11pt]{amsart}

\usepackage{xypic}
\usepackage{color}

\allowdisplaybreaks[4]

\newtheorem{proposition}{Proposition}[section]
\newtheorem{lemma}[proposition]{Lemma}
\newtheorem{corollary}[proposition]{Corollary}
\newtheorem{theorem}[proposition]{Theorem}

\theoremstyle{definition}

\theoremstyle{remark}

\newcommand{\thlabel}[1]{\label{th:#1}}
\newcommand{\thref}[1]{Theorem~\ref{th:#1}}
\newcommand{\selabel}[1]{\label{se:#1}}
\newcommand{\seref}[1]{Section~\ref{se:#1}}
\newcommand{\lelabel}[1]{\label{le:#1}}
\newcommand{\leref}[1]{Lemma~\ref{le:#1}}
\newcommand{\prlabel}[1]{\label{pr:#1}}
\newcommand{\prref}[1]{Proposition~\ref{pr:#1}}
\newcommand{\colabel}[1]{\label{co:#1}}

\newcommand{\eqlabel}[1]{\label{eq:#1}}
\newcommand{\equref}[1]{(\ref{eq:#1})}

\def\equal#1{\smash{\mathop{=}\limits^{#1}}}

\newcommand{\Hom}{{\rm Hom}}

\newcommand{\Ker}{{\rm Ker}\,}

\newcommand{\im}{{\rm Im}\,}

\def\lan{\langle}
\def\ran{\rangle}
\def\ot{\otimes}

\def\rightact{\hbox{$\leftharpoonup$}}
\def\leftact{\hbox{$\rightharpoonup$}}

\newcommand{\Cc}{\mathcal{C}}

\newcommand{\Mm}{\mathcal{M}}

\def\*C{{}^*\hspace*{-1pt}{\Cc}}

\def\text#1{{\rm {\rm #1}}}

\def\ol{\overline}
\def\ul{\underline}

\begin{document}
\title[Weak bialgebras and monoidal categories]{Weak bialgebras and monoidal categories}
\author{G. B\"ohm}
\address{Research Institute for particle and nuclear Physics, P.O.B. 49,
H-1525\break Budapest 114, Hungary}
\email{g.bohm@rmki.kfki.hu}
\urladdr{http://www.rmki.kfki.hu/\~{}bgabr/}
\author{S. Caenepeel}
\address{Faculty of Engineering, 
Vrije Universiteit Brussel, B-1050 Brussels, Belgium}
\email{scaenepe@vub.ac.be}
\urladdr{http://homepages.vub.ac.be/\~{}scaenepe/}
\author{K. Janssen}
\address{Faculty of Engineering,
Vrije Universiteit Brussel, B-1050 Brussels, Belgium}
\email{krjanssen@hotmail.com}
\urladdr{http://homepages.vub.ac.be/\~{}krjansse/}
\thanks{The first named author was supported by the Hungarian
Scientific Research Fund OTKA K68195.The second and third named authors were supported by the research project G.0117.10  
``Equivariant Brauer groups and Galois deformations'' from
FWO-Vlaanderen. All authors wish to thank Joost
Vercruysse for fruitful discussions. } 
\dedicatory{To Mia Cohen, on the occasion of her retirement.}



\begin{abstract}
We study monoidal structures on the category of (co)modules over a weak
bialgebra. Results due to Nill and Szlach\'anyi are unified and extended to
infinite algebras. We discuss the coalgebra structure on the source and target
space of a weak bialgebra.
\end{abstract}
\maketitle
\section*{Introduction}
Let $H$ be a $k$-algebra. It is well-known that bialgebra structures on $H$
are in bijective correspondence to monoidal structures on the category
${}_H\Mm$ of left $H$-modules that are fibred, which means that the forgetful
functor ${}_H\Mm\to {}_k\Mm$ is strongly monoidal.\\ 
Weak bialgebras are more general than bialgebras. The axioms $\Delta(1)= 1\ot
1$ and $\varepsilon(hk)=\varepsilon(h)\varepsilon(k)$ are replaced by four
weaker axioms (lm), (rm), (lc), (rc). Weak bialgebras and weak Hopf algebras
were introduced in \cite{BMS}, although the history of the subject goes back
to earlier notions, we refer to the introduction in \cite{BMS} for more
detail. For a survey on weak bialgebras (and Hopf algebras), we refer to
\cite{NV}.\\ 
A natural question that arises is whether the category of representations of a
weak bialgebra is monoidal. An indirect answer to this question is given by
the observation that a weak bialgebra  $H$  is a left bialgebroid
 over a separable Frobenius algebra which is denoted by $H_t$ in the
sequel,  see for example \cite[Sec. 3.2]{B}.  
Szlach\'anyi \cite{Sz} has reformulated the definition of a left bialgebroid
in terms of monoidal categories, implying that the category of representations
of a weak bialgebra is indeed monoidal  and the forgetful functor to the
module category of the ground ring possesses a so-called separable Frobenius
structure. The product on the category of modules over a weak bialgebra is
given by the tensor product over   
$H_t$. 
In \cite{Pastro}, it is shown that the (co)module category of any weak
bialgebra in an idempotent complete braided monoidal category is
monoidal with a separable Frobenius forgetful functor. More generally, in
\cite{BLS} those monads (termed weak bimonads) were described whose
Eilenberg-Moore category is monoidal such that the forgetful functor possesses
a separable Frobenius structure.   \\  
The question  about monoidality of the (co)module category of a weak
bialgebra has also been investigated by Nill \cite{Nill}, see also
\cite{BMS2}  and \cite{NV}.  Actually it suffices that only two of
the four axioms, namely (rm) and (lm) are fulfilled, and then the
representation category is monoidal. But this time the product is defined in a
different way, as a subspace  (or quotient)  of the usual tensor
product  (but no longer a (co)module tensor product),   
and the unit object is now a subspace of the dual of the prebialgebra. 
 Recall that  Nill only looks at finite dimensional weak Hopf
algebras, and this condition is needed in his result. This has the obvious
advantage that the self-duality of the axioms can be fully exploited, and
allows to conclude immediately that the category of comodules is monoidal once
(rc) and (lc) are fulfilled.\\ 
The aim of this note is to unify and extend the above mentioned results, in a
straightforward and elementary way. We look at weak bialgebras over
commutative rings that are not necessary finite. Then we can show that
conditions (lm) and (rm) imply that the representation category is monoidal,
while (lc) and (rc) imply that the corepresentation category is monoidal.
 Furthermore, we prove that, under these assumptions, the respective
forgetful functor is both monoidal and opmonoidal. It becomes essentially
strong (op)monoidal if any third one of the listed axioms holds. 
In this case the monoidal product in the (co)representation category is given 
by a module tensor product or, isomorphically, by a comodule tensor
product. The forgetful functor is of the separable Frobenius type if
and only if all of the four listed axioms hold.    \\
Now let $H$ be weak bialgebra, that is, all four axioms are satisfied. Then it
can be shown that the unit object of the representation category is isomorphic
to $H_t$, the subalgebra of $H$ that we mentioned above  and that 
was observed by  Szlach\'anyi \cite{Sz}  to be  a separable
Frobenius algebra. General properties of Frobenius algebras then show that
$H_t$ is also a coseparable Frobenius coalgebra, and  they can be used to
obtain an alternative proof of the fact that  the the tensor product on
the representation category is actually the  comodule tensor product
 over $H_t$ and is isomorphic to the tensor product over $H_t$. Similar
results can be proved for the corepresentation category.\\  
The paper is organized as follows. In \seref{1}, we revisit the elementary
properties of weak bialgebras, where we indicate which axioms are needed for
each property. 
We will not give full detail on all the proofs, but we have tried to organize
our text in such a way that the missing details can be easily filled
in. Sections \ref{se:2} and \ref{se:3} are devoted to resp. the categories of
modules and of comodules. In \seref{Frob}, we fist recall some general theory
on separable Frobenius algebras, and then we apply it to weak bialgebras.\\
We will use standard terminology and notation from classical Hopf algebra
theory, see for example \cite{Mont} or \cite{DascalescuNR}. In particular, we
use the Sweedler notation for comultiplications and coactions. In the Sweedler
indices, we use brackets $()$ for comultiplication and square brackets $[]$
for coactions. If several copies of $\Delta(1)$ occur in the same formula,
then we use primes to distinguish between them: we write $\Delta(1)=1_{(1)}\ot
1_{(2)}=1_{(1')}\ot 1_{(2')}$. For the general theory of monoidal 
categories, we refer to \cite{K}.

\section{Weak bialgebras}\selabel{1}
Let $k$ be a commutative ring, and $H$ a $k$-module carrying the structure of
$k$-algebra and $k$-coalgebra, with unit $1$, comultiplication $\Delta$ and
counit $\varepsilon$, and assume that $\Delta(hk)=\Delta(h)\Delta(k)$, for all
$h,k\in H$. Then we call $H$ a $k$-prebialgebra. $H$ is called finite if $H$
is finitely generated and projective as a $k$-module. If we work over a field,
then this means that $H$ is finite dimensional.\\
A prebialgebra is called left monoidal (resp. right monoidal, left comonoidal,
right comonoidal) if condition (lm) (resp. (rm), (lc), (rc)) holds, for all
$h,k,l\in H$: 
$$\begin{array}{lcl}
{\rm (lm)} ~~~\varepsilon(hkl)=
\varepsilon(hk_{(1)})\varepsilon(k_{(2)}l)&;&
 \Delta^2(1)=1_{(1)}\ot 1_{(2)}1_{(1')}\ot 1_{(2')}~~~{\rm (lc)}\\
{\rm (rm)} ~~~\varepsilon(hkl)=
\varepsilon(hk_{(2)})\varepsilon(k_{(1)}l)&;&
 \Delta^2(1)=1_{(1)}\ot 1_{(1')}1_{(2)}\ot 1_{(2')}~~~{\rm (rc)}
\end{array}$$
A prebialgebra that is at the same time left and right (co)monoidal is called
(co)monoidal, and then we say that condition (m) (resp. condition (c)) is
satisfied. (l) (resp. (r)) will mean that $H$ is a the same time left
(resp. right) monoidal and comonoidal. A weak bialgebra is a monoidal and
comonoidal prebialgebra. This terminology agrees with the terminology in
\cite{BMS}. What we call a prebialgebra is called a weak bialgebra in
\cite{Nill}.\\ 
Let $H^*$ be the linear dual of $H$. $H^*$ is a $k$-algebra, with the opposite
convolution as multiplication: $\lan h^*k^*,h\ran= \lan h^*,h_{(2)}\ran \lan
k^*,h_{(1)}\ran$. If $H$ is finite, then $H^*$ is a $k$-coalgebra, with
comultiplication given by the formula $\Delta(h^*)=h^*_{(1)}\ot h^*_{(2)}$ if
and only if $\lan h^*,hk\ran= \lan h^*_{(1)},k\ran\lan h^*_{(2)},h\ran$, for
all $h,k\in H$. So if $H$ is finite, then the dual of prebialgebra is again a
prebialgebra. 

\begin{lemma}\lelabel{00}
Let $H$ be a finite prebialgebra. $H$ satisfies (lm) (resp. (rm)) if and only
if $H^*$ satisfies (lc) (resp. (rc)). 
\end{lemma}

\begin{proof}
We will prove that $H^*$ satisfies (lc) if $H$ satisfies (lm). All the other
implications can be proved in a similar way. $H^*$ satisfies (lc) if and only
if 
$$\Delta^2(\varepsilon)=\varepsilon_{(1)}\ot
\varepsilon_{(2)}\varepsilon_{(1')}\ot \varepsilon_{(2')}.$$
This is equivalent to $\lan \varepsilon,hkl\ran$ being equal to
\begin{eqnarray*}
\lan \varepsilon_{(1)}, l\ran \lan \varepsilon_{(2)}\varepsilon_{(1')}, k\ran
\lan \varepsilon_{(2')}, h\ran
&=& \lan \varepsilon_{(1)}, l\ran \lan \varepsilon_{(2)},k_{(2)}\ran
\lan \varepsilon_{(1')},k_{(1)}\ran \lan \varepsilon_{(2')}, h\ran\\
&=& \lan \varepsilon, k_{(2)}l\ran \lan \varepsilon, hk_{(1)}\ran,
\end{eqnarray*}
for all $h,k,l\in H$, and this is clearly satisfied if $H$ satisfies (lm).
\end{proof}

If $H$ is a prebialgebra, then
$H$ is an $H^*$-bimodule, and $H^*$ is an $H$-bimodule:
$$h^*\leftact h \rightact k^*= \lan k^*,h_{(1)}\ran h_{(2)} \lan
h^*,h_{(3)}\ran~~;~~ \lan h\leftact h^*\rightact k,l\ran =\lan
h^*,klh\ran.$$  
Now consider the maps $f,f':\ H^*\to H~~;~~g,g':\ H\to H^*$,
$$\begin{array}{cc}
f(h^*)=1\rightact h^*=\lan h^*,1_{(1)}\ran 1_{(2)}&
f'(h^*)=h^*\leftact 1=\lan h^*,1_{(2)}\ran 1_{(1)}\\
g(h)= h\leftact \varepsilon=\lan \varepsilon,-h\ran&
g'(h)= \varepsilon\rightact h=\lan \varepsilon,h-\ran
\end{array}$$
If $H$ is finite, then we can identify $H$ and $H^{**}$, and then $f'=f^*$ and
$g'=g^*$. If we specify in our notation that these maps depend on the
prebialgebra $H$, that is, we write $f=f_H$, $g=g_H$, etc., then $f_{H^*}=g_H$
and $f'_{H^*}=g'_{H}$ if $H$ is finite. Now we define eight more maps 
$\varepsilon_t,\varepsilon_s,\ol{\varepsilon}_t,\ol{\varepsilon}_s:\ H\to H$ and
$\varepsilon^*_t,\varepsilon^*_s,\ol{\varepsilon}^*_t,\ol{\varepsilon}^*_s:\ H^*\to H^*$. 
$$\begin{array}{cccc}
\varepsilon_t=f\circ g&\varepsilon_s=f'\circ g'&
\ol{\varepsilon}_t=f\circ g'&\ol{\varepsilon}_s=f'\circ g\eqlabel{1.1.10}\\
\varepsilon^*_t=g\circ f&\varepsilon^*_s=g'\circ f'&
\ol{\varepsilon}^*_t=g\circ f'&\ol{\varepsilon}^*_s=g'\circ f\eqlabel{1.1.12}
\end{array}$$
We have that $(\varepsilon_t)^*=\varepsilon^*_s$,
$(\varepsilon_s)^*=\varepsilon^*_t$,
$(\ol{\varepsilon}_t)^*=\ol{\varepsilon}^*_t$ and
$(\ol{\varepsilon}_s)^*=\ol{\varepsilon}^*_s$. The explicit description of the
maps is as follows:
$$\begin{array}{cc}
\varepsilon_t(h)=\lan \varepsilon,1_{(1)}h\ran 1_{(2)}&
\varepsilon_s(h)=\lan \varepsilon,h1_{(2)}\ran 1_{(1)}\eqlabel{1.1.13}\\
\ol{\varepsilon}_t(h)=\lan \varepsilon,h1_{(1)}\ran 1_{(2)}&
\ol{\varepsilon}_s(h)=\lan \varepsilon,1_{(2)}h\ran 1_{(1)}\eqlabel{1.1.14}
\end{array}$$

\begin{lemma}\lelabel{1.1}
Let $H$ be a prebialgebra. Then we have the following properties, for all
$h,g\in H$:
\begin{eqnarray}
{\rm (rm)}& \Longleftrightarrow &
 h\varepsilon_t(g)=\lan \varepsilon,h_{(1)}g\ran h_{(2)}\eqlabel{1.1.1}\\
{\rm (rm)}& \Longleftrightarrow &
 \varepsilon_s(g)h=\lan \varepsilon,gh_{(2)}\ran h_{(1)}\eqlabel{1.1.2}\\
{\rm (lm)}& \Longleftrightarrow &
 \ol{\varepsilon}_t(g)h=\lan \varepsilon,gh_{(1)}\ran h_{(2)}\eqlabel{1.1.3}\\
{\rm (lm)}& \Longleftrightarrow &
h\ol{\varepsilon}_s(g)=\lan \varepsilon,h_{(2)}g\ran h_{(1)}\eqlabel{1.1.4}\\
{\rm (rc)}& \Longleftrightarrow &
 h_{(1)}\ot \varepsilon_t(h_{(2)})=1_{(1)}h\ot 1_{(2)}\eqlabel{1.1.5}\\
{\rm (rc)}& \Longleftrightarrow &
 \varepsilon_s(h_{(1)})\ot h_{(2)}=1_{(1)}\ot h1_{(2)}\eqlabel{1.1.6}\\
{\rm (lc)}& \Longleftrightarrow &
h_{(1)}\ot \ol{\varepsilon}_t(h_{(2)})=h1_{(1)}\ot 1_{(2)}\eqlabel{1.1.7}\\
{\rm (lc)}& \Longleftrightarrow &\ol{\varepsilon}_s(h_{(1)})\ot
h_{(2)}=1_{(1)}\ot 1_{(2)}h\eqlabel{1.1.8} 
\end{eqnarray}
\end{lemma}

We introduce several subspaces of $H$ and $H^*$.

$$\begin{array}{cccc}
\im(f)=H_L&\im(f')=H_R&\im(g)=H_L^*&\im(g')=H_R^*\\
\im(\varepsilon_t)=H_t&\im(\varepsilon_s)=H_s&\im(\varepsilon^*_t)=H^*_t&\im(\varepsilon^*_s)=H^*_s\\  
\im(\ol{\varepsilon}_t)=\ol{H}_t&\im(\ol{\varepsilon}_s)=\ol{H}_s&\im(\ol{\varepsilon}^*_t)=\ol{H}^*_t&\im(\ol{\varepsilon}^*_s)=\ol{H}^*_s\\ 
\end{array}$$
$$
\begin{array}{cc}
I_t=\{h\in H~|~\Delta(h)=1_{(1)}h\ot 1_{(2)}\}&
I_s=\{h\in H~|~\Delta(h)=1_{(1)}\ot h1_{(2)}\}\\
\ol{I}_t=\{h\in H~|~\Delta(h)=h1_{(1)}\ot 1_{(2)}\}&
\ol{I}_s=\{h\in H~|~\Delta(h)=1_{(1)}\ot 1_{(2)}h\}
\end{array}$$
$$\begin{array}{c}
I^*_t=\{h^*\in H^*~|~\lan h^*,kl\ran=\lan\varepsilon,kl_{(2)}\ran\lan
h^*,l_{(1)}\ran\}\\ 
I^*_s=\{h^*\in H^*~|~\lan h^*,kl\ran=\lan\varepsilon,k_{(1)}l\ran\lan
h^*,k_{(2)}\ran\}\\ 
\ol{I}^*_t=\{h^*\in H^*~|~\lan h^*,kl\ran=\lan\varepsilon,kl_{(1)}\ran\lan
h^*,l_{(2)}\ran\}\\ 
\ol{I}^*_s=\{h^*\in H^*~|~\lan h^*,kl\ran=\lan\varepsilon,k_{(2)}l\ran\lan
h^*,k_{(1)}\ran\} 
\end{array}
$$

\begin{lemma}\lelabel{1.2}
Let $H$ be a prebialgebra. Then
$$\begin{array}{ccc}
I_t\subset H_t\subset H_L\supset \ol{H}_t\supset \ol{I}_t&;&
I_s\subset H_s\subset H_R\supset \ol{H}_s\supset \ol{I}_s\\
I^*_t\subset H^*_t\subset H^*_L\supset \ol{H}^*_t\supset \ol{I}^*_t&;&
I^*_s\subset H^*_s\subset H^*_R\supset \ol{H}^*_s\supset \ol{I}^*_s
\end{array}$$
\begin{eqnarray}
{\rm (rc)}&\Longrightarrow&I_t=H_t=H_L;~~I_s=H_s=H_R;~~I_t^*=H_t^*;~~I_s^*=H_s^*;\eqlabel{1.2.1}\\
{\rm (lc)}&\Longrightarrow&\ol{I}_t=\ol{H}_t=H_L;~~\ol{I}_s=\ol{H}_s=H_R;~~
\ol{I}_t^*=\ol{H}_t^*;~~\ol{I}_s^*=\ol{H}_s^*;\eqlabel{1.2.2}\\
{\rm (rm)}&\Longrightarrow&I_t^*=H_t^*=H_L^*;~~I_s^*=H_s^*=H_R^*;~~I_t=H_t;~~I_s=H_s;\eqlabel{1.2.3}\\
{\rm (lm)}&\Longrightarrow&\ol{I}_t^*=\ol{H}_t^*={H}_L^*;~~\ol{I}_s^*=\ol{H}_s^*={H}_R^*;~~
\ol{I}_t=\ol{H}_t;~~\ol{I}_s=\ol{H}_s.\eqlabel{1.2.4}
\end{eqnarray}
Consequently
\begin{eqnarray}
{\rm (c)}&\Longrightarrow& H_t=\ol{H}_t;~~H_s=\ol{H}_s;\eqlabel{1.2.7}\\
{\rm (m)}&\Longrightarrow& H^*_t=\ol{H}^*_t;~~H^*_s=\ol{H}^*_s\eqlabel{1.2.6}.
\end{eqnarray}
\end{lemma}

\begin{proof}
We first prove the statements relating $I_t$, $H_t$ and $H_L$.
If $h\in I_t$, then $\varepsilon_t(h)=\varepsilon(1_{(1)}h)1_{(2)}=\varepsilon(h_{(1)})h_{(2)}=h$,
hence $h\in H_t$.\\
Obviously $H_t=\im(f\circ g)\subset \im (f)=H_L$.\\
Assume that (rc) holds, and take $h=f(h^*)=\lan h^*,1_{(1)}\ran 1_{(2)}\in H_L=\im(f)$.
Then $\Delta(h)=\lan h^*,1_{(1)}\ran 1_{(1')}1_{(2)}\ot 1_{(2')}=
1_{(1')}h\ot 1_{(2')}$, hence $h\in I_t$.\\
Now assume that (rm) holds, and take $h=\varepsilon_t(k)\in H_t$. 
Then
$\Delta(h)=\varepsilon(1_{(1)}k)\Delta(1_{(2)})=\varepsilon(1_{(1)}k)1_{(2)}\ot
1_{(3)} \equal{\equref{1.1.1}} 1_{(1)}\varepsilon_t(k)\ot 1_{(2)}=1_{(1)}h\ot 1_{(2)}$,
and it follows that $h\in I_t$.\\
If $H$ is finite, then the statements relating $I^*_t$, $H^*_t$ and $H^*_L$ follow
by applying the statements relating $I_t$, $H_t$ and $H_L$ to the prebialgebra $H^*$.
They can be proved easily without the finiteness assumption.
First assume that $h^*\in I_t^*$. Then for all $h\in H$, we have that
$\lan \varepsilon_t^*(h^*),h\ran=\lan h^*,\varepsilon_s(h)\ran=
\lan \varepsilon,h1_{(2)}\ran\lan h^*,1_{(1)}\ran =\lan h^*,h1\ran =\lan h^*,h\ran$, hence
$\varepsilon_t^*(h^*)=h^*\in H_t^*$.\\
Obviously $H_t^*=\im(g\circ f)\subset \im(g)=H_L^*$.\\
Now assume that (rm) holds, and take $h^*\in \im(g)=H_L^*$. Then there exists
$h\in H$ such that $\lan h^*,k\ran=\lan\varepsilon,kh\ran$, for all $k\in H$. Then we have
for all $k,l\in H$ that
$\lan h^*,kl\ran=\lan \varepsilon,klh\ran\equal{{\rm (rm)}}
\lan \varepsilon, kl_{(2)}\ran \lan \varepsilon, l_{(1)}h\ran=
\lan \varepsilon, kl_{(2)}\ran \lan h^*,l_{(1)}\ran$
and this means that $h^*\in I_t^*$.\\
Now assume that (rc) holds, and take $\varphi\in H_t^*$. This means that there
exists $h^*\in H^*$ such that $\lan \varphi,h\ran=\lan h^*,\varepsilon_s(h)\ran=
\lan \varepsilon, h1_{(2)}\ran \lan h^*,1_{(1)}\ran$,
and then we have for all $k,l\in H$ that
$$
\lan\varphi, kl\ran= \lan \varepsilon, kl1_{(2)}\ran \lan h^*,1_{(1)}\ran
\equal{\equref{1.1.6}}\lan \varepsilon,kl_{(2)}\ran\lan h^*,\varepsilon_s(l_{(1)})\ran
=\lan \varepsilon,kl_{(2)}\ran\lan \varphi,l_{(1)}\ran,
$$
and this implies that $\varphi\in I_t^*$. The proof of all the other inclusions is similar.
\end{proof}

Note that, in fact, in the cases of \equref{1.2.3} and  \equref{1.2.4}, also
the opposite implications $I_t^*=H_L^*\Rightarrow \text{(rm)} \Leftarrow I_s^*=H_R^*$
and $\ol I_t^*=H_L^*\Rightarrow \text{(lm)} \Leftarrow \ol
I_s^*=H_R^*$ hold.\\
Recall that a $k$-module $M$ is called locally projective if for each finite subset
$\{m_1,\cdots, m_N\}\subset M$, there exist finite subsets $\{e^*_1,\cdots,e^*_m\}\subset M^*$
and $\{e_1,\cdots,e_m\}\subset M$ such that $m_i=\sum_{j=1}^m \lan e_j^*,m_i\ran e_j$, for all
$ i\in \{1,\cdots,N\}$. It is then easy to show that the natural map
$\alpha:\ M\ot N\to \Hom(M^*,N)$, $\alpha(m\ot n)(m^*)=\lan m^*,m\ran n$ is injective, for every
$k$-module $N$,
in other words, $M$ satisfies the $\alpha$-condition. Applying this property, it is easy to
show that, in the case where $H$ is locally projective,
$I_t=H_L\Rightarrow \text{(rc)} \Leftarrow I_s=H_R$ and $\ol
I_t=H_L\Rightarrow \text{(lc)} \Leftarrow \ol I_s=H_R$. 

\begin{lemma}\lelabel{1.3}
Let $H$ be a prebialgebra.
If (rm) or (rc) holds, then the maps $\varepsilon_t$, $\varepsilon_s$,
$\varepsilon_t^*$ and $\varepsilon_s^*$ are idempotent; if (lm) or (lc) holds,
then the maps $\ol{\varepsilon}_t$, $\ol{\varepsilon}_s$,
$\ol{\varepsilon}_t^*$ and $\ol{\varepsilon}_s^*$ are idempotent.
\end{lemma}

\begin{lemma}\lelabel{lem:eps_1}
In any prebialgebra $H$, the following identities hold.
\begin{eqnarray}
&\varepsilon_s(1_{(1)}) \ot 1_{(2)}=1_{(1)}\ot \varepsilon_t(1_{(2)});\quad
&\ol\varepsilon_s(1_{(1)}) \ot 1_{(2)}=1_{(1)}\ot \ol\varepsilon_t(1_{(2)});\eqlabel{eq:eps_1}\\
&\varepsilon_t(1_{(1)}) \ot 1_{(2)}=1_{(2)}\ot \ol\varepsilon_t(1_{(1)});\quad
&\ol\varepsilon_s(1_{(2)}) \ot 1_{(1)}=1_{(1)}\ot \varepsilon_s(1_{(2)}) \eqlabel{eq:eps_1_tw}.
\end{eqnarray}
\end{lemma}

\begin{lemma}\lelabel{1.4}
Let $H$ be a prebialgebra.
\begin{eqnarray*}
{\rm (rm)}&\Longrightarrow& \varepsilon_t(h)\varepsilon_t(g)=\varepsilon_t(\varepsilon_t(h)g)
~{\rm and}~\varepsilon_s(h)\varepsilon_s(g)=\varepsilon_s(h\varepsilon_s(g))
\eqlabel{1.4.1}\\
{\rm (lm)}&\Longrightarrow&
\ol{\varepsilon}_t(h)\ol{\varepsilon}_t(g)=\ol{\varepsilon}_t(h\ol{\varepsilon}_t(g))
~{\rm and}~\ol{\varepsilon}_s(h)\ol{\varepsilon}_s(g)=\ol{\varepsilon}_s(\ol{\varepsilon}_s(h)g)
\eqlabel{1.4.2}\\
{\rm (rc)}&\Longrightarrow&
\varepsilon^*_t(h^*)\varepsilon^*_t(g^*)=\varepsilon^*_t(\varepsilon^*_t(h^*)g^*)
~{\rm and}~\varepsilon^*_s(h^*)\varepsilon^*_s(g^*)=\varepsilon^*_s(h^*\varepsilon^*_s(g^*))
\eqlabel{1.4.3}\\
{\rm (lc)}&\Longrightarrow&
\ol{\varepsilon}^*_t(h^*)\ol{\varepsilon}^*_t(g^*)=\ol{\varepsilon}^*_t(h^*\ol{\varepsilon}^*_t(g^*))
~{\rm and}~\ol{\varepsilon}^*_s(h^*)\ol{\varepsilon}^*_s(g^*)=\ol{\varepsilon}^*_s(\ol{\varepsilon}^*_s(h^*)g^*) 
\eqlabel{1.4.4}
\end{eqnarray*}
Consequently $H_t$ and $H_s$ are subalgebras of $H$ if (rm) holds;
$\ol{H}_t$ and $\ol{H}_s$ are subalgebras if (lm) holds; $H^*_t$ and $H^*_s$
are subalgebras of $H^*$ if (rc) holds; $\ol{H}^*_t$ and $\ol{H}^*_s$ are subalgebras of $H^*$ if (lc) holds. 
\end{lemma}

\begin{proof}
Assume that (rm) holds; then $I_t=H_t$, and we have for all $h,g\in H$ that
$\varepsilon_t(h)\varepsilon_t(g)\equal{\equref{1.1.1}}\lan\varepsilon,\varepsilon_t(h)_{(1)}g\ran
\varepsilon_t(h)_{(2)}=\lan\varepsilon,1_{(1)}\varepsilon_t(h)g\ran 1_{(2)}=
\varepsilon_t(\varepsilon_t(h)g)$.\\
Now assume that (rc) holds.
Then
\begin{eqnarray*}
&&\hspace*{-2cm} \lan \varepsilon_t^*(g^*)\varepsilon_t^*(h^*),h\ran=
\lan \varepsilon_t^*(g^*), h_{(2)}\ran \lan \varepsilon_t^*(h^*), h_{(1)}\ran\\
&=&\lan \varepsilon_t^*(g^*), h_{(2)}\ran \lan h^*,\varepsilon_s(h_{(1)})\ran
\equal{\equref{1.1.6}}\lan \varepsilon_t^*(g^*), h1_{(2)}\ran \lan h^*, 1_{(1)}\ran\\
&\equal{(*)}& 
\lan \varepsilon, h1_{(3)}\ran \lan \varepsilon_t^*(g^*), 1_{(2)}\ran \lan h^*, 1_{(1)}\ran
=\lan \varepsilon, h1_{(2)}\ran \lan \varepsilon_t^*(g^*)h^*, 1_{(1)}\ran\\
&=& \lan \varepsilon_t^*(g^*)h^*, \varepsilon_s(h)\ran = \lan \varepsilon_t^*(\varepsilon_t^*(g^*)h^*),h\ran.
\end{eqnarray*}
$(*)$: we used that $\varepsilon_t^*(g^*)\in I_t^*$, see \equref{1.2.1}. The proof of all the other
assertions is similar.
\end{proof}

There are more distinguished subalgebras than those in the previous
lemma:
\begin{lemma}\lelabel{lem:more_algs}
Let $H$ be a prebialgebra.
\begin{eqnarray*}
{\rm (rc)}&\Longrightarrow&H_s~{\rm and}~H_t~{\rm are~subalgebras~of}~H\\
{\rm (lc)}&\Longrightarrow&\ol H_s~{\rm and}~\ol H_t~{\rm are~subalgebras~of}~ H\\
{\rm (rm)}&\Longrightarrow&H^*_s~{\rm and}~H^*_t~{\rm are~subalgebras~of}~H^* \\
{\rm (lm)}&\Longrightarrow&\ol H^*_s~{\rm and}~\ol H^*_t~{\rm are~subalgebras~of}~H^*
\end{eqnarray*}
\end{lemma}
\begin{proof}
Assume first that (rc) holds. Then for any $h,k\in H$, 
\begin{eqnarray*}
&&\hspace*{-15mm}\varepsilon_t(h)\varepsilon_t(k)=
\lan \varepsilon,1_{(1)} k\ran\varepsilon_t(h)1_{(2)}
\equal{\equref{1.1.6}}
\lan\varepsilon,\varepsilon_s(\varepsilon_t(h)_{(1)})k\ran\varepsilon_t(h)_{(2)}\\
&\equal{\equref{1.2.1}}&\lan\varepsilon,\varepsilon_s(1_{(1)}\varepsilon_t(h))k\ran
1_{(2)}=
\lan \varepsilon, 1_{(1')}k\ran\lan \varepsilon,1_{(1'')}h\ran\lan\varepsilon,1_{(1)}1_{(2'')}1_{(2')}\ran1_{(2)}\\
&\equal{\rm{(rc)}}&\lan\varepsilon,1_{(1')}k\ran\lan\varepsilon,1_{(1)}h\ran\lan\varepsilon,1_{(2)}1_{(2')}\ran 1_{(3)}=
\lan\varepsilon,\varepsilon_s(1_{(2)})k\ran\lan\varepsilon,1_{(1)}h\ran1_{(3)}\\
&=&\lan(\varepsilon_t^*\circ g)(k),1_{(2)}\ran\lan g(h),1_{(1)}\ran 1_{(3)}=
f\big((\varepsilon_t^*\circ g)(k)g(h)\big).
\end{eqnarray*}
This is an element of $H_L$, hence by \equref{1.2.1} an element of
$H_t$. Assume next that (rm) holds. Then for any $h^*,k^*\in H^*$ and $h\in H$, 
\begin{eqnarray*}
&&\hspace*{-15mm}\lan \varepsilon^*_t(h^*)\varepsilon^*_t(k^*),h\ran=
\lan h^*,\varepsilon_s(h_{(2)})\ran\lan k^*,\varepsilon_s(h_{(1)})\ran\\
&=&\lan h^*,\varepsilon_s(h_{(2)})\ran\lan k^*,1_{(1)}\ran \lan
\varepsilon,h_{(1)}1_{(2)}\ran\equal{\equref{1.1.1}}
\lan h^*,\varepsilon_s(h\varepsilon_t(1_{(2)})\ran\lan k^*,1_{(1)}\ran\\
&=&\lan h^*,1_{(1')}\ran\lan\varepsilon,h\varepsilon_t(1_{(2)})1_{(2')}\ran\lan k^*,1_{(1)}\ran
=\lan\varepsilon,h\,(\varepsilon_t\circ f)(k^*)\,f(h^*)\ran\\
&=&\lan g\big((\varepsilon_t\circ f)(k^*)f(h^*)\big),h\ran.
\end{eqnarray*}
This shows that $\varepsilon^*_t(h^*)\varepsilon^*_t(k^*)$ is an element of
$H^*_L$, hence by \equref{1.2.3} an element of $H^*_t$. The rest of the proof is analogous.
\end{proof}

If (rm) or (rc) holds, then $\varepsilon_s$ and $\varepsilon_t$ are projections, hence
$H_s$ and $H_t$ are direct summands of $H$, and $H_s\ot H_t$ is a direct summand
of $H\ot H$; moreover $H_s\ot H_t=\im(\varepsilon_s\ot \varepsilon_t)$. In a similar
way, if (lm) or (lc) holds, then $\ol{H}_s\ot \ol{H}_t=\im(\ol{\varepsilon}_s\ot \ol{\varepsilon}_t)$.
If (rc) holds, then we easily show that $(\varepsilon_s\ot\varepsilon_t)(\Delta(1))=\Delta(1)$,
so that we have the following result.

\begin{lemma}\lelabel{1.5}
Let $H$ be a prebialgebra.
\begin{equation}\eqlabel{1.5.1}
{\rm (rc)}\Longrightarrow \Delta(1)\in H_s\ot H_t~~~;~~~
{\rm (lc)}\Longrightarrow \Delta(1)\in \ol{H}_s\ot \ol{H}_t.
\end{equation}
\end{lemma}

The dual version of \leref{1.5} is slightly more involved. If (rm) or (rc) holds, then
$H_s^*$ and $H_t^*$ are direct summands of $H^*$, $H_s^*\ot H_t^*$ is a direct summand
of $H^*\ot H^*$, and $H_s^*\ot H_t^*=\im (\varepsilon_s^*\ot \varepsilon_t^*)
= \im(g'\circ f'\ot g\circ f)\subset \im(g'\ot g)$. The converse implication also holds
if (rm) is satisfied: take $g'(h)\ot g(k)\in \im(g'\ot g)$. By \equref{1.2.3}, there exist
$h^*,k^*\in H^*$ such that $g'(h)=g'(f'(h^*))$ and $g(k)=g(f(k^*))$. Then
$g'(h)\ot g(k)=g'(f'(h^*))\ot g(f(k^*))\in  \im(g'\circ f'\ot g\circ f)$.\\
\leref{locproj} is probably folklore, we include an elementary proof for the sake of
completeness.

\begin{lemma}\lelabel{locproj}
If $H^*$ is locally projective as a $k$-module (this is automatically
satisfied in the cases where $k$ is a field or $H$ is finite),
then the map  
$$\iota:\ H^*\ot H^*\to (H\ot H)^*,~~\lan \iota(h^*\ot k^*),k\ot h\ran =\lan h^*,h\ran
\lan k^*,k\ran$$
is injective.
\end{lemma}

\begin{proof}
Assume that $\iota(\sum_{\alpha} h_\alpha^*\ot k_\alpha^*)=0$. Since $H^*$ is
locally projective, there exist $f_i\in H^{**}$ and $k_i\in H^*$ such that
$h_\alpha^*=\sum_i f_i(h_\alpha^*) k_i$, for each $\alpha$. For all $h,k\in H$, we have
that
$\sum_\alpha \lan k_\alpha^*,k\ran \lan h_\alpha^*,h\ran =0$,
hence we have for all $k\in H$ that
$\sum_\alpha \lan k_\alpha^*,k\ran  h_\alpha^*=0$.
Then it follows that we have for every index $i$ and for every $k\in H$ that
$$\sum_\alpha f_i(h_\alpha^*)\lan k_\alpha^*,k\ran= f_i\bigl(
\sum_\alpha \lan k_\alpha^*,k\ran  h_\alpha^*\bigr)=0.$$
Therefore $\sum_\alpha f_i(h_\alpha^*) k_\alpha^*=0$
and
$$\sum_\alpha h_\alpha^*\ot k_\alpha^*= \sum_ik_i\ot \Bigl(\sum_\alpha f_i(h_\alpha^*) k_\alpha^*\Bigr)
=0.$$
\end{proof}

Assume that $H^*$ is locally projective, and let
$$H^o=\{h^*\in H^*~|~m^*(h^*)=h^*\circ m\in \im(\iota)\},$$
where $m$ is the multiplication in $H$.
For $h^*\in H^o$, we write $\iota^{-1}(m^*(h^*))=\Delta(h^*)=
h^*_{(1)}\ot h^*_{(2)}$. Then $\Delta$ is coassociative, see 
for example \cite[Prop. 1.5.3]{DascalescuNR} (at least in the case where $k$ is
a field). It is easy to show that $H^o$ is closed under opposite convolution multiplication.

\begin{proposition}\prlabel{1.6}
Let $H$ be a prebialgebra, and assume that $H^*$ is locally projective as
a $k$-module. If (rm) holds, then $\varepsilon\in H^o$, 
$\Delta(\varepsilon)\in H^*_s\ot H^*_t$ and $H^o$ is a prebialgebra satisfying
(rc). $H_s^*$ and $H_t^*$ are subspaces of $H^o$, and $\varepsilon_s^*$ and
$\varepsilon_t^*$ restrict to projections from $H^o$ onto $H_s^*$ and $H_t^*$.\\
If (rc) holds, then $\varepsilon\in H^o$, 
$\Delta(\varepsilon)\in \ol{H}^*_s\ot \ol{H}^*_t$ and $H^o$ is a weak prebialgebra satisfying
(lc). $\ol{H}_s^*$ and $\ol{H}_t^*$ are subspaces of $H^o$, and $\ol{\varepsilon}_s^*$ and
$\ol{\varepsilon}_t^*$ restrict to projections from $H^o$ onto $\ol{H}_s^*$ and $\ol{H}_t^*$.\\
If $H$ is a weak bialgebra, then $H^o$ is also a weak bialgebra.
\end{proposition}

\begin{proof} 
If (rm) holds, then
$\lan \varepsilon, hk\ran =\lan \varepsilon, h1_{(2)})\ran\lan \varepsilon,1_{(1)}1_{(2')}\ran
\lan \varepsilon,1_{(1')}k\ran$,
hence
$$
\Delta(\varepsilon)= \lan \varepsilon,1_{(1)}1_{(2')}\ran
g'(1_{(1 ')})\ot g(1_{(2)}) \in \im(g'\ot g)= H_s^*\ot H_t^*.$$
Take $\psi\in H_s^*=I_s^*$ (see \equref{1.2.3}). Then 
$$\lan \psi,kl\ran =\lan \varepsilon,k_{(1)}l\ran\lan\psi,k_{(2)}\ran=
\lan \varepsilon_{(1)},l\ran \lan\psi\varepsilon_{(2)},k\ran,$$
hence
$\Delta(\psi)=\varepsilon_{(1)}\ot\psi\varepsilon_{(2)}$ and $\psi\in H^o$. Then we compute that
$(\Delta\ot H^o)(\varepsilon_{(1)}\ot \varepsilon_{(2)})= \varepsilon_{(1')}\ot
\varepsilon_{(1)}\varepsilon_{(2')}\ot \varepsilon_{(2)}$, which proves that $H^o$
satisfies (rc).
\end{proof}

\begin{lemma}\lelabel{1.7}
If (rm) or (rc) holds, then we have isomorphisms
$$\xymatrix{H_t\ar[r]<3pt>^{g}&H_t^*\ar[l]<3pt>^{f}}~~{\rm and}~~
\xymatrix{H_s\ar[r]<3pt>^{g'}&H_s^*\ar[l]<3pt>^{f'}}.$$
If (lm) or (lc) holds, then we have isomorphisms
$$\xymatrix{\ol{H}_t\ar[r]<3pt>^{g'}&\ol{H}_s^*\ar[l]<3pt>^{f}}~~{\rm and}~~
\xymatrix{\ol{H}_s\ar[r]<3pt>^{g}&\ol{H}_t^*\ar[l]<3pt>^{f'}}.$$
\end{lemma}

\begin{proof}
Assume that (rm) or (rc) holds. We first show that $g$ restricts and corestricts to
$g:\ H_t\to H_{ t}^*$. Take $z\in H_t$. Since $f\circ
g=\varepsilon_t$ is a projection, we have that $g(z)=(g\circ f\circ
g)(z)=\varepsilon_t^*(g(z))\in H_t^*$.\\ 
In a similar way, $f$ restricts and corestricts to $f:\ H_t^*\to H_t$: for $\varphi\in H_t^*$,
we have $f(\varphi)=(f\circ g\circ f)(\varphi)=\varepsilon_t(f(\varphi))\in H_t$.\\
Finally, for $z\in H_t$ and $\varphi\in H_t^*$, we have $(f\circ g)(z)=\varepsilon_t(z)=z$
and $(g\circ f)(\varphi)=\varepsilon_t^*(\varphi)=\varphi$.
\end{proof}

\begin{lemma}\lelabel{1.8}
If (rm) or (rc) holds, then we have isomorphisms
$H_s^*\cong (H_t)^*$ and $H_t^*\cong (H_s)^*$. Moreover $H_s$ and $H_t$ are
finite with finite dual basis $1_{(1)}\ot g(1_{(2)})$ for
$H_s$ and $1_{(2)}\ot g'(1_{(1)})$ for $H_t$. 
\\
If (lm) or (lc) holds, then
$\ol{H}_t^*\cong (\ol{H}_t)^*$ and $\ol{H}_s^*\cong (\ol{H}_s)^*$.
$\ol{H}_s$ and $\ol{H}_t$ are finite, with finite dual basis $1_{(1)}\ot g'(1_{(2)})$ for
$\ol{H}_s$ and $1_{(2)}\ot g(1_{(1)})$ for $\ol{H}_t$. 
\end{lemma}

\begin{proof}
Recall that $H_s^*=\im(\varepsilon_s^*)$ and $(H_t)^*$ is the dual of $H_t$.
If (rm) or (rc) holds, then $\varepsilon_t$ is a projection, and $H=H_t\oplus K_t$,
with $K_t=\Ker(\varepsilon_t)$. Take $h^*\in H_s^*$. Then $h^*=\varepsilon_s^*(h^*)=
h^*\circ \varepsilon_t$, hence $h^*(K_t)=0$.
Conversely, if $h^*(K_t)=0$, then
for all $h\in H$, we have $\lan h^*,h\ran =\lan h^*,\varepsilon_t(h)\ran +
\lan h^*,h-\varepsilon_t(h)\ran= \lan h^*,\varepsilon_t(h)\ran=
\lan \varepsilon_s^*(h^*),h\ran$, hence $h^*=\varepsilon_s^*(h^*)\in H_s^*$. We conclude
that $h^*\in H_s^*$ if and only if $h^*(K_t)=0$.
Now define
$\alpha:\ (H_t)^*\to H_s^*$, $\alpha(z^*)=z^*\circ \varepsilon_t$, and
$\beta:\ H_s^*\to (H_t)^*$, $\beta(\varphi)=\varphi_{|H_t}$. $\alpha$ and $\beta$ are
well-defined and inverse to each other.\\
If (rc) holds, then $1_{(1)}\ot 1_{(2)}\in H_s\ot H_t$, by \leref{1.5}, and then
$1_{(2)}\ot g'(1_{(1)})\in H_t\ot H_s^*\cong H_t\ot (H_t)^*$, by \leref{1.7}.
If (rm) holds, then $\varepsilon_t(1_{(2)})\ot g'(1_{(1)})\equal{\equref{eq:eps_1}}
1_{(2)}\ot\lan\varepsilon,\varepsilon_s(1_{(1)}) - \ran\equal{\equref{1.1.2}}
1_{(2)}\ot g'(1_{(1)})$. Using the fact that $H_R^*=H_s^*$, see \equref{1.2.3},
and that $\varepsilon_t$ and $\varepsilon_s^*$ are projections onto $H_t$
and $H_s^*$, by \leref{1.3}, we now easily obtain $(\varepsilon_t\ot
\varepsilon^*_s)(1_{(2)}\ot g'(1_{(1)})) = 1_{(2)}\ot g'(1_{(1)})$, hence
$1_{(2)} \ot g'(1_{(1)})\in \Ker({\rm id}-\varepsilon_t\ot \varepsilon^*_s)
=H_t\ot H_s^*$. 
Finally, we have for all $z\in H_t$ that $1_{(2)}\lan g'(1_{(1)}),z\ran=
1_{(2)}\lan \varepsilon,1_{(1)}z\ran=\varepsilon_t(z)=z$.
\end{proof}

\begin{lemma}\lelabel{1.9}
Let $H$ be a prebialgebra.
\begin{eqnarray}
{\rm (lm)}&\Longrightarrow& g\circ \ol{\varepsilon}_s=g,~g'\circ \ol{\varepsilon}_t=g',~
\varepsilon_t\circ \ol{\varepsilon}_s=\varepsilon_t~~{\rm and}~~
\varepsilon_s\circ \ol{\varepsilon}_t=\varepsilon_s;\eqlabel{1.9.1}\\
{\rm (rm)}&\Longrightarrow& g' \circ {\varepsilon}_s=g',~g \circ {\varepsilon}_t=g,~
\ol{\varepsilon}_t\circ {\varepsilon}_s=\ol{\varepsilon}_t~~{\rm and}~~
\ol{\varepsilon}_s\circ {\varepsilon}_t=\ol{\varepsilon}_s;\eqlabel{1.9.2}\\
{\rm (lc)}&\Longrightarrow& f\circ \ol{\varepsilon}^*_s=f,~
f'\circ \ol{\varepsilon}^*_t=f',~
\varepsilon^*_t\circ \ol{\varepsilon}^*_s=\varepsilon^*_t~~{\rm and}~~
\varepsilon^*_s\circ \ol{\varepsilon}^*_t=\varepsilon^*_s;\\
{\rm (rc)}&\Longrightarrow& f'\circ {\varepsilon}^*_s=f',~
f\circ {\varepsilon}^*_t=f,~
\ol{\varepsilon}^*_t\circ {\varepsilon}^*_s=\ol{\varepsilon}^*_t~~{\rm and}~~
\ol{\varepsilon}^*_s\circ {\varepsilon}^*_t=\ol{\varepsilon}^*_s.
\end{eqnarray}
\end{lemma}

\begin{proof}
If (lm) holds, then  applying $\varepsilon$ to both sides of \equref{1.1.4} we obtain 
$g= g\circ \ol{\varepsilon}_s$. 
Then it follows that $\varepsilon_t\circ \ol{\varepsilon}_s= 
f\circ g\circ \ol{\varepsilon}_s=f\circ g=\varepsilon_t$. The proof of all the other assertions is
similar.
\end{proof}

The proof of our next lemma is straightforward.

\begin{lemma}\lelabel{1.10}
Let $H$ be a prebialgebra.
\begin{eqnarray}
{\rm (c)}&\Longrightarrow& \varepsilon_s(h)\varepsilon_t(k)=\varepsilon_t(k)\varepsilon_s(h);
\eqlabel{1.10.1}\\
{\rm (m)}&\Longrightarrow&\varepsilon_s(h_{(1)})\ot \varepsilon_t(h_{(2)})=
\varepsilon_s(h_{(2)})\ot \varepsilon_t(h_{(1)}).\eqlabel{1.10.2}
\end{eqnarray}
\end{lemma}

\begin{proposition}\prlabel{1.11}
If (m) holds, then we have  anti-algebra  isomorphisms
$$\xymatrix{H_t\ar[r]<3pt>^{\ol{\varepsilon}_s}&\ol{H}_s\ar[l]<3pt>^{\varepsilon_t}}~~{\rm and}~~
\xymatrix{H_s\ar[r]<3pt>^{\ol{\varepsilon}_t}&\ol{H}_t\ar[l]<3pt>^{\varepsilon_s}}.$$
If (c) holds, then we have  anti-algebra  isomorphisms
$$\xymatrix{H_s^*\ar[r]<3pt>^{\ol{\varepsilon}_t^*}&\ol{H}_t^*\ar[l]<3pt>^{\varepsilon_s^*}}~~{\rm and}~~
\xymatrix{H_t^*\ar[r]<3pt>^{\ol{\varepsilon}_s^*}&\ol{H}_s^*\ar[l]<3pt>^{\varepsilon_t^*}}.$$
\end{proposition}

\begin{proof}
Assume that (m) holds. Then $H_t^*=\ol{H}_t^*$, see \equref{1.2.6}. According
to \leref{1.7}, we have isomorphisms
$$\xymatrix{H_t\ar[r]<3pt>^(.4){g}&H_t^*=\ol{H}_t^*\ar[l]<3pt>^(.6){f}
\ar[r]<3pt>^(.6){f'}&\ol{H}_s\ar[l]<3pt>^(.4){g}}
$$
Then it suffices to observe that $f'\circ g=\ol{\varepsilon}_s$ and $f\circ g=\varepsilon_t$.
 Moreover, for any $y,y'\in \ol{H}_s$, 
\begin{eqnarray*}
&&\hspace{-2cm}\varepsilon_t(y)\varepsilon_t(y')\equal{(*)}
\varepsilon_t(\varepsilon_t(y)y')=
\lan\varepsilon,1_{(1)}1_{(2')}y'\ran\lan\varepsilon,1_{(1')}y\ran 1_{(2)}\\
&\equal{\equref{1.2.4}}&
\lan\varepsilon,1_{(1)}y'_{(2)}\ran\lan\varepsilon,y'_{(1)}y\ran
1_{(2)}\equal{\rm (rm)} 
\lan\varepsilon,1_{(1)}y'y\ran1_{(2)}=
\varepsilon_t(y'y),
\end{eqnarray*}
where we applied \leref{1.4} at $(*)$.\\
If (c) holds then for any $\varphi,\varphi'\in \ol H_t^*$ and $z\in H_t$,
\begin{eqnarray*}
&&\hspace{-2cm}\lan\varepsilon_s^*(\varphi)\varepsilon_s^*(\varphi'),z\ran
\equal{(*_1)}
\lan\varepsilon_s^*(\varphi\varepsilon_s^*(\varphi')),z\ran\\
&\equal{(*_2)}&\lan\varphi,z_{(2)}\ran\lan\varphi',\varepsilon_t(z_{(1)})\ran=
\lan\varphi,z_{(2)}\ran\lan\varphi',1_{(2)}\ran\lan\varepsilon,1_{(1)}z_{(1)}\ran\\
&\equal{(*_3)}&\lan\varphi,1_{(1)}z\ran\lan\varphi',1_{(2)}\ran
\equal{(*_4)}\lan\varphi,z_{(1)}\ran\lan\varphi',z_{(2)}\ran\\
&=&\lan\varphi'\varphi,z\ran\equal{(*_2)}
\lan\varepsilon_s^*(\varphi'\varphi),z\ran .
\end{eqnarray*}
$(*_1)$: we applied \leref{1.4}; $(*_2)$: $z\in H_t$; $(*_3)$: $\varphi\in\ol H_t^*=\ol I_t^*$;
$(*_4)$: $z\in H_t=I_t$.
Since by \leref{1.4} $\varepsilon_s^*(\varphi)\varepsilon_s^*(\varphi')$ is an
element of $H^*_s\cong (H_t)^*$, this proves multiplicativity of
$\varepsilon_s^*:\ol H_t^*\to H^*_s$. It is an isomorphism by similar
considerations as in the first part of the proof.

\end{proof}

\begin{lemma}
For a monoidal prebialgebra $H$ and any elements $h,k\in H$, the following
identities hold. 
\begin {eqnarray}
&\ol\varepsilon_s(h)\varepsilon_t(k)=\varepsilon_t(k)\ol\varepsilon_s(h)\qquad
&1_{(1)}\ol\varepsilon_s(h)\ot 1_{(2)}=1_{(1)}\ot 1_{(2)}\varepsilon_t(h)
\eqlabel{eq:sbar_t}\\
&\varepsilon_s(h)\ol\varepsilon_t(k)=\ol \varepsilon_t(k)\varepsilon_s(h)\qquad
&\varepsilon_s(h)1_{(1)}\ot 1_{(2)}=1_{(1)}\ot \ol\varepsilon_t(h)1_{(2)}.
\eqlabel{eq:tbar_s}
\end{eqnarray}
\end{lemma}
\begin{proof}
We prove only \equref{eq:sbar_t}, \equref{eq:tbar_s} is checked
symmetrically. As for the first equality, 
\begin{eqnarray*}
\ol\varepsilon_s(h)\varepsilon_t(k)&=&
1_{(1)}\varepsilon_t(k)\lan\varepsilon,1_{(2)}h\ran\equal{\equref{1.1.1}}
\lan\varepsilon,1_{(1)}k\ran 1_{(2)}\lan\varepsilon,1_{(3)}h\ran\\
&\equal{\equref{1.1.4}}&\lan\varepsilon,1_{(1)}k\ran 1_{(2)}\ol\varepsilon_s(h)=
\varepsilon_t(k)\ol\varepsilon_s(h).
\end{eqnarray*}
The second equality follows by
$$
1_{(1)}\ol\varepsilon_s(h)\ot 1_{(2)}\equal{\equref{1.1.4}}
\lan\varepsilon,1_{(2)}h\ran1_{(1)}\ot 1_{(3)}\equal{\equref{1.1.1}}
1_{(1)}\ot 1_{(2)}\varepsilon_t(h).
$$
\end{proof}

Let $A$ be a finite algebra, with finite dual basis $\sum_i a_i\ot a_i^*\in
A\ot A^*$. It is well-known that $A^*$ carries a coalgebra structure, given by
the formula $$\Delta(a^*)=\sum_{i,j} \lan a^*,a_ia_j\ran a_j^*\ot a_i^*.$$

\begin{proposition}\prlabel{1.12}
 If any of (rm) and (rc) holds, then $H_s$ and $H_t$, and also $H^*_s$
and $H^*_t$ carry coalgebra structures. If any of (lm) and (lc) holds, then
$\ol H_s$ and $\ol H_t$, and also $\ol H^*_s$ and $\ol H^*_t$ carry coalgebra
structures.   
The comultiplication maps are given by the formulas
$$\begin{array}{cc}
\Delta_s(y)=  \varepsilon_s  (1_{(1)})\ot \varepsilon_s(y1_{(2)})&
\Delta_s^*(\psi)=  \varepsilon_s^*(g'(1_{(1)}))\ot \varepsilon_s^*(\psi g(1_{(2)}))  
\\
\Delta_t(z)= \varepsilon_t(1_{(1)}z)\ot  \varepsilon_t  (1_{(2)})&
\Delta_t^*(\varphi)=  \varepsilon_t^*(g'(1_{(1)})\varphi)\ot \varepsilon_t^*(g(1_{(2)}))
\\
\ol{\Delta}_s(y)= \ol \varepsilon_s  (1_{(1)})\ot\ol{\varepsilon}_s(1_{(2)}y)&
\ol{\Delta}_s^*(\psi)=  \ol \varepsilon_s^*(\psi g'(1_{(2)}))\ot\ol\varepsilon_s^*(g(1_{(1)}))
\\
\ol{\Delta}_t(z)=\ol{\varepsilon}_t(z1_{(1)})\ot  \ol \varepsilon_t (1_{(2)})&
\ol{\Delta}_t^*(\varphi)= \ol\varepsilon_t^*(g'(1_{(2)})) \ot \ol\varepsilon_t^*(g(1_{(1)})\varphi)
\end{array}$$
\end{proposition}

\begin{proof}
Let $H$ be a 
prebialgebra satisfying (rm)  or (rc).  From Lemmas \ref{le:1.4},
\ref{le:lem:more_algs}  and \ref{le:1.8}, we know that $H_s$ and $H_t$ are
finitely generated projective $k$-algebras, hence $(H_s)^*\cong H_t^*$ and
$(H_t)^*\cong H_s^*$ are $k$-coalgebras. The other statements are obtained in
a similar way.\\ 
The comultiplication maps can be computed easily using the finite dual bases given \leref{1.8}.
For example, if  (rm) or  (rc) holds, then $1_{(2)}\ot
g'(1_{(1)} )\in H_t\ot H_s^*\cong H_t\ot (H_t)^*$ is a finite dual basis for
$H_t$, hence the comultiplication on $H_t$ is given by
\begin{eqnarray*}
\Delta_t(z)&=&\lan g'(1_{(1)})g'(1_{(1')}),z\ran 1_{(2')}\ot 1_{(2)}\\
&=&  \varepsilon_t(z_{(1)})\ot \varepsilon_t(z_{(2)})\stackrel{H_t=I_t}=
\varepsilon_t(1_{(1)}z) \ot \varepsilon_t(1_{(2)}).
\end{eqnarray*}
 The comultiplication on $H^*_s$ is defined as the transpose of the
opposite multiplication in $H_t$. That is, for all $z,z'\in H_t$ and $\psi \in
H^*_s\cong (H_t)^*$, 
\begin{eqnarray*}
&&\hspace{-2cm}\lan\Delta_s^*(\psi),z'\ot z\ran=
\lan\psi,zz'\ran=
\lan\varepsilon,z_{(1)}z'_{(1)}\ran\lan \psi,z_{(2)}z'_{(2)}\ran\\
&\stackrel{H_t=I_t}=&\lan\varepsilon,z_{(1)}z'\ran\lan \psi,z_{(2)}\ran\stackrel{z'\in H_t}=
\lan\varepsilon,1_{(1)}z'\ran\lan\varepsilon,z_{(1)}1_{(2)}\ran\lan
\psi,z_{(2)}\ran\\
&=&\lan g'(1_{(1)})\ot \psi g(1_{(2)}),z'\ot z\ran\\
&\stackrel{z,z'\in H_t}=&\lan \varepsilon_s^*(g'(1_{(1)}))\ot
\varepsilon_s^*(\psi g(1_{(2)})),z'\ot z\ran .
\end{eqnarray*}

\end{proof}

\begin{proposition}\prlabel{coalg1}
Let $H$ be a prebialgebra satisfying  (rm) or  (rc). Then
$\varepsilon_t:\ H\to H_t$ and $\varepsilon_s:\ H\to H_s$ are coalgebra
maps. In a similar way, if  (lm) or  (lc) is satisfied, then
$\ol{\varepsilon}_s$ and $\ol{\varepsilon}_t$ are coalgebra maps.\\
 If (m) holds then 
$\xymatrix @C=15pt {
H_t\ar@<1.5pt>[r]^-{\ol \varepsilon_s}&
\ol H_s\ar@<1.5pt>[l]^-{\varepsilon_t}}$ 
and 
$\xymatrix @C=15pt {
H_s\ar@<1.5pt>[r]^-{\ol \varepsilon_t}&
\ol H_t\ar@<1.5pt>[l]^-{\varepsilon_s}}$
are anti-algebra and coalgebra isomorphisms.\\

If (c) holds, then $\ol{H}_t=H_t$, $\ol{H}_s=H_s$, $\ol{\Delta}_t=\Delta_t^{\rm cop}$ and
$\ol{\Delta}_s=\Delta_s^{\rm cop}$. 
\end{proposition}

\begin{proof}
 If (rm) holds, then we have  for all $h\in H$ that
\begin{eqnarray*}
&&\hspace*{-2cm}\varepsilon_t(h_{(1)})\ot \varepsilon_t(h_{(2)})=
\varepsilon_t(h_{(1)})\ot\lan\varepsilon,1_{(1)}h_{(2)}\ran 1_{(2)}\stackrel{\equref{1.1.2}}=
\varepsilon_t(\varepsilon_s(1_{(1)})h)\ot 1_{(2)}\\
&\equal{\equref{eq:eps_1}}&\varepsilon_t(1_{(1)}h)\ot \varepsilon_t(1_{(2)})
=\lan\varepsilon, 1_{(1')}1_{(1)}h\ran 1_{(2')}\ot \varepsilon_t(1_{(2)})\\ 
&\equal{\equref{1.1.1}}&
\varepsilon_t(1_{(1)}\varepsilon_t(h))\ot \varepsilon_t(1_{(2)})=
\Delta_t(\varepsilon_t(h)).
\end{eqnarray*}
If (rc) holds, then $\varepsilon_t$ is idempotent, by \leref{1.3}, and we have
for all $h\in H$ that
$
\varepsilon_t(h_{(1)})\ot \varepsilon_t(h_{(2)})\stackrel{\equref{1.1.5}}=
\varepsilon_t(1_{(1)}h)\ot \varepsilon_t(1_{(2)})=
\Delta_t(\varepsilon_t(h))$.
\\
Assume that (m) holds. Then the stated maps are anti-algebra isomorphisms by
\prref{1.11}. Moreover, 
the diagram
$$\xymatrix{
&H\ar[dl]_{\varepsilon_s}\ar[dr]^{\ol{\varepsilon}_t}&\\
H_s\ar[rr]^{\ol{\varepsilon}_t}&& H_t}$$
commutes, by \equref{1.9.2}. Now the two vertical maps in the diagram are coalgebra maps,
hence the horizontal map is also a coalgebra map.\\
 Finally,  assume that (c) holds. Then
\begin{eqnarray*}
&&\hspace{-2cm}\ol{\Delta}_t(z)=\ol{\varepsilon}_t(z1_{(1)})\ot 1_{(2)}= 
\lan \varepsilon,z1_{(1)}1_{(1')}\ran 1_{(2')}\ot 1_{(2)}\\
&=&\lan \varepsilon,1_{(1)}1_{(1')}z\ran 1_{(2')}\ot 1_{(2)}=
1_{(2')}\ot {\varepsilon}_t(1_{(1')}z)=
\Delta^{\rm cop}_t(z),
\end{eqnarray*}
 where the third equality follows by \equref{1.5.1} and
\equref{1.10.1}. 
\end{proof}

\begin{corollary}\colabel{coalg1b}
Let $H$ be a prebialgebra satisfying (rm), and assume that $H^*$ is locally projective.
Then $\varepsilon_t^*:\ H^o\to H_t^*$ and $\varepsilon_s^*:\ H^o\to H_s^*$ are
coalgebra maps.
\end{corollary}

\begin{proof}
We have seen in \prref{1.6} that $H^o$ is a prebialgebra satisfying (rc). Then it suffices
to apply \prref{coalg1} to $H^o$.
\end{proof}

\section{The representation category}\selabel{2}
In this section, we follow the constructions from \cite{Nill}.
Let $H$ be a prebialgebra, and take $M,N\in {}_H\Mm$. $M\ot N$ is a left $H\ot H$-module,
and then an associative but non-unital left $H$-module via restriction of
scalars via $\Delta$: 
$$h\cdot (m\ot n)= h_{(1)}m\cdot\otimes h_{(2)}\cdot n.$$
This $H$-action is non-unital in general, since we do not assume the property
$\Delta(1)=1\ot 1$. However, the $H$-action becomes unital if we restrict $M\ot N$
to $M\ot^l N=1\cdot (M\ot N)$. Then $\ot^l$ is an associative tensor product on
${}_H\Mm$, the associativity constraint is induced by the natural associativity
constraint on the category ${}_k\Mm$ of $k$-modules. In order to make ${}_H\Mm$
a monoidal category, we need a unit object.\\
Observe that $H_L^*=\im(g)$ is a left $H$-module. Indeed, an element of $\varphi\in H_L^*$ is
of the form $\varphi=h\leftact \varepsilon$, for some $h\in H$, and we define
$$k\cdot\varphi=(kh)\leftact \varepsilon=g(kh)\in H_L^*.$$
This $H$-action on $H_L^*$ is defined in such a way that $g:\ H\to H_L^*$ is left $H$-linear.
$H_L^*$ will be our candidate for unit object.
Take $M\in {}_H\Mm$, and consider the maps
\begin{eqnarray*}
l_M:\ H_L^*\ot M\to M&;&l_M(\varphi\ot m)=f(\varphi)m;\\
\ol{l}_M:\ M\to H_L^*\ot^l M&;& \ol{l}_M(m)=g(1_{(1)})\ot 1_{(2)}m;\\
r_M:\ M\ot H_L^*\to M&;& r_M(m\ot \varphi)=f'(\varphi)m;\\
\ol{r}_M:\ M\to M\ot^l H_L^*&;&\ol{r}_M(m)=1_{(1)}m\ot g(1_{(2)}).
\end{eqnarray*}
The restrictions of $l_M$ and $r_M$ will be the candidates for the left and
right unit constraints, and $\ol{l}_M$ and $\ol{r}_M$ for their inverses. For
all $m\in M$, we have that
\begin{equation}\eqlabel{2.0.1}
(l_M\circ \ol{l}_M)(m)=(r_M\circ \ol{r}_M)(m)=m.
\end{equation}

\begin{proposition}\prlabel{2.1}
For a prebialgebra $H$, the following conditions are equivalent:
\begin{enumerate}
\item $\ol{l}_M$ (resp. $\ol{r}_M$) is left $H$-linear, for all $M\in {}_H\Mm$;
\item $\ol{l}_H$ (resp. $\ol{r}_H$) is left $H$-linear;
\item for all $h\in H$, we have that $g(1_{(1)})\ot 1_{(2)}h=g(h_{(1)})\ot h_{(2)}$
(resp. $1_{(1)}h\ot g(1_{(2)})=h_{(1)}\ot g(h_{(2)})$);
\item (lm) (resp. (rm)) holds.
\end{enumerate}
\end{proposition}

\begin{proof}
We only prove the ``left'' part.
$\ul{1)\Rightarrow 2)}$ is trivial.\\
$\ul{2)\Rightarrow 3)}$. If $\ol{l}_H$ is left $H$-linear, then we have for all $h\in H$ that
$g(1_{(1)})\ot 1_{(2)}h=\ol{l}_H(h)=h\cdot\ol{l}_H(1)=
h_{(1)}g(1_{(1)})\ot h_{(2)}1_{(2)}1= g(h_{(1)})\ot h_{(2)}$,
and condition 3) follows.\\
$\ul{3)\Rightarrow 1)}$.
$\ol{l}_M(hm)=g(1_{(1)})\ot 1_{(2)}hm=
g(h_{(1)})\ot h_{(2)}m=h\ol{l}_M(m)$.\\
$\ul{3)\Rightarrow 4)}$. Evaluating the first tensor factor of $g(1_{(1)})\ot 1_{(2)}h=g(h_{(1)})\ot h_{(2)}$
at $k\in H$, we obtain
\begin{equation}\eqlabel{2.2.1}
\lan \varepsilon, k1_{(1)}\ran1_{(2)}h=\lan \varepsilon, kh_{(1)}\ran h_{(2)}.
\end{equation}
Apply $\varepsilon$ to the second tensor factor of \equref{2.2.1}:
\begin{equation}\eqlabel{2.2.2}
\lan \varepsilon, k1_{(1)}\ran\lan\varepsilon,1_{(2)}h\ran=\lan \varepsilon, kh\ran.
\end{equation}
(lm) follows after we
multiply \equref{2.2.1} to the right by $l\in H$, and apply $\varepsilon$ to it:
$$\lan\varepsilon,khl\ran\equal{\equref{2.2.2}}
\lan\varepsilon,k1_{(1)}\ran\lan\varepsilon,1_{(2)}hl\ran=
\lan\varepsilon,kh_{(1)}\ran\lan\varepsilon,h_{(2)}l\ran.$$
$\ul{4)\Rightarrow 3)}$. For all $h\in H$, we have that
\begin{eqnarray*}
&&\hspace*{-2cm}
g(h_{(1)})\ot h_{(2)}\equal{\equref{1.9.1}}
g(\ol{\varepsilon}_s(h_{(1)}))\ot h_{(2)}=
g(1_{(1)})\lan \varepsilon,1_{(2)}h_{(1)}\ran \ot h_{(2)}\\
&\equal{\equref{1.1.3}}& g(1_{(1)})\ot \ol{\varepsilon}_t(1_{(2)})h\equal{ \equref{eq:eps_1}}
g(\ol{\varepsilon}_s(1_{(1)}))\ot 1_{(2)}h\equal{\equref{1.9.1}} g(1_{(1)})\ot 1_{(2)}h.
\end{eqnarray*}
\end{proof}

\begin{proposition}\prlabel{2.3}
Let $H$ be a monoidal prebialgebra. Then $\ol{l}_M:\ M\to H_L^*\ot^l M$
and $\ol{r}_M:\ M\to M\ot^l H_L^*$ are bijective, and their inverses
are respectively the restriction of $l_M$ to $H_L^*\ot^l M$ and the
restriction of $r_M$ to $M\ot^l H_L^*$.
\end{proposition}

\begin{proof}
$H_L^*\ot^l M$ is generated by elements of the form $g(1_{(1)}h)\ot 1_{(2)}m$.
We first compute that
$$l_M(g(1_{(1)}h)\ot 1_{(2)}m)= (f\circ g)(1_{(1)}h)1_{(2)}m=
\varepsilon_t(1_{(1)}h)1_{(2)}m=\varepsilon_t(h)m.$$
then
\begin{eqnarray*}
&&\hspace*{-2cm}
\ol{l}_M(\varepsilon_t(h)m)=g(1_{(1)})\ot 1_{(2)}\varepsilon_t(h)m
=
g(\varepsilon_t(h)_{(1)})\ot \varepsilon_t(h)_{(2)}m\\
&\equal{\equref{1.2.3}}&
g(1_{(1)}\varepsilon_t(h))\ot 1_{(2)}m
=  g(1_{(1)}h)\ot 1_{(2)}m, 
\end{eqnarray*}
 where the last equality follows by the identity
$g(hk)=\varepsilon(-hk)\equal{\equref{1.1.1}}
  \varepsilon(-h\varepsilon_t(k))=g(h\varepsilon_t(k))$, for all $h,k\in H$. 
This proves that $\ol{l}_M$ is a left inverse of $l_M$ and we already know from \equref{2.0.1}
that it is a right inverse.
\end{proof}

\begin{theorem}\thlabel{2.4}
Let $H$ be a monoidal prebialgebra.
Then we have a monoidal category $({}_H\Mm,\ot^l,H_L^*,a,l,r)$.
\end{theorem}

\begin{proof}
It is clear that $\ol{l}$ and $\ol{r}$ are natural, and then the restrictions of
$l$ and $r$ are also natural, by \prref{2.3}. To finish the proof, it suffices to show
that the triangle axiom \cite[XI.(2.9)]{K} holds. This can be seen as follows:
for $M,N\in {}_H\Mm$ and $1_{(1)}m\ot 1_{(2)}n\in M\ot^l N$, we have that
$
\ol{r}_M(1_{(1)}m)\ot 1_{(2)}n=
1_{(1')}1_{(1)}m\ot g(1_{(2')})\ot 1_{(2)}n
= 1_{(1)}m\ot g(1_{(2)})\ot 1_{(3)}n
= 1_{(1)}m\ot  g  (1_{(1')})\ot 1_{(2')}1_{(2)}n
1_{(1)}m\ot \ol{l}_N(1_{(2)}n)$.
\end{proof}

The map $\pi:\ M\ot N\to M\ot^l N$, $\pi (m\ot n)=1_{(1)}m\ot 1_{(2)}n$ is
a projection. Define $M\ot_l N=M\ot N/\Ker({\rm id}-\pi)$. Then the map
$$\ol{\pi}:\ M\ot_l N\to M\ot^l N,~~\ol{\pi}([m\ot n])=\pi(m\ot n)$$
is a well-defined isomorphism of $k$-modules. The left $H$-action on
$M\ot^l N$ can be transported to a left $H$-action on $M\ot_l N$, which is
given by the usual formula $h\cdot [m\ot n]=[h_{(1)}m\ot h_{(2)}n]$.
This means that the functors $\ot^l$ and $\ot_l$ are isomorphic, so - in the case
where $H$ is monoidal - the
monoidal structure on ${}_H\Mm$ can be described by $\ot_l$.\\
 
From Lemmas \ref{le:1.2} and \ref{le:1.7}, we know that if $H$ is a monoidal
prebialgebra then $H_L^*=H_t^*$ and $f:\ H_t^*\to H_t$ is an isomorphism with
inverse $g$. So we can transport the left $H$-action on $H_L^*$ to $H_t$,
making $H_t$ the unit object of ${}_H\Mm$. This transported action 
is easily computed: for $h\in H$, $z\in H_t$, we find $h.
z=f(h\cdot
g(z))=(f\circ g)(hz)=\varepsilon_t(hz)$.\\ 
Recall that any monoidal functor $U:(\mathcal M,\ot, I)\to (\mathcal N,
\times, J)$ factorizes uniquely through a monoidal functor from $\Mm$ to the
bimodule category of the monoid $U(I)$ in $\mathcal N$ via the forgetful
functor ${}_{U(I)}\mathcal N_{U(I)}\to \mathcal N$. In \cite{Sz05} a monoidal
functor $U$ was termed {\em essentially strong monoidal} whenever the functor
$\mathcal M \to {}_{U(I)}\mathcal N_{U(I)}$ in its factorization is
strong monoidal. There is an evident dual notion of {\em essentially strong
opmonoidal} functor. 

\begin{theorem}\thlabel{thm:forgetful}
Let $H$ be a monoidal prebialgebra. Then the forgetful functor $U:{}_H\mathcal
M\to {}_k\Mm$ is monoidal and opmonoidal. If any of (rc) and (lc) holds,
then it is essentially strong monoidal and  essentially strong opmonoidal.
\end{theorem}

\begin{proof}
The monoidal structure of $U$ is given by 
$$
U^{M,N}:   M\ot^l N  \xymatrix{\ \ar@{>->}[r]&M\ot N};\qquad 
U^0:\xymatrix{H^*_L\ar[r]^-{\lan-,1\ran}&k}
$$

and the opmonoidal structure of $U$ is given by 
$$
U_{M,N}:\xymatrix{M\ot N\ar[r]^-\pi&M\ot^l N};\qquad 
U_0:\xymatrix{k\ar[r]^-{(-)\varepsilon}&H_L^*}
$$
for any left $H$-modules $M$ and $N$. Note that $U_{M,N}\circ U^{M,N}$ is the
identity map. These monoidal and opmonoidal structures render 
$U$ with the structure of a separable Frobenius monoidal functor if and only
if $H$ is a weak bialgebra. \\ 
The monoidal structure of $U$ induces $H_t$-bimodule actions on any left
$H$-module $M$. Their explicit form comes out as
$$
z\triangleright m = z.m\qquad \text{and}\qquad 
m\triangleleft z=\ol\varepsilon_s(z).m.
$$
(Associativity of the right action follows directly by \prref{1.11} and both
actions commute in light of \equref{eq:sbar_t}). The opmonoidal structure of
$U$ induces $H_t$-bicomodule coactions on any left $H$-module $M$. Explicitly, 
$$
m\mapsto 1_{(1)}.m\ot\varepsilon_t(1_{(2)})\qquad \text{and}\qquad 
m\mapsto \varepsilon_t(1_{(1)}) \ot 1_{(2)}.m.
$$
(Commutativity of these coactions can be verified also directly by 
\begin{eqnarray*}
&&\hspace{-15mm}\varepsilon_t(1_{(1)})\ot 1_{(2)}1_{(1')}\ot\varepsilon_t(1_{(2')})=
1_{(2)}\ot\ol\varepsilon_t(1_{(1)})\varepsilon_s(1_{(1')})\ot 1_{(2')}\\
&\equal{\equref{eq:tbar_s}}&
1_{(2)}\ot\varepsilon_s(1_{(1')})\ol\varepsilon_t(1_{(1)})\ot 1_{(2')}=
\varepsilon_t(1_{(1)})\ot1_{(1')}1_{(2)}\ot\varepsilon_t(1_{(2')}),
\end{eqnarray*}
where the first and last equalities follow by \leref{lem:eps_1}. By similar
steps also coassociativity of the coactions can be checked directly.)\\
The monoidal functor $U:{}_H\Mm\to {}_k\Mm$ factorizes through a monoidal
functor $S:{}_H\Mm\to {}_{H_t}\Mm_{H_t}$. The binary part of the monoidal
structure of $S$ comes out as  
$$
S^{M,N}:M\ot_{H_t} N\to M\ot^l N,\qquad m \ot_{H_t} n \mapsto 1_{(1)}.m \ot 1_{(2)}.n .
$$
(One can check directly that it is well-defined by \equref{eq:sbar_t}.)
In order to prove essentially strong monoidality of $U$, we need to construct
the inverse of $S^{M,N}$.
As a candidate, consider the restriction $S_{M,N}:M\ot^l N\to M\ot_{H_t} N$
of the canonical epimorphism $M\ot N\to M\ot_{H_t} N$,
$$
S_{M,N}\bigl(1_{(1)}.m \ot 1_{(2)}.n\bigr)= 1_{(1)}.m \ot_{H_t} 1_{(2)}.n.
$$
Clearly, for any monoidal prebialgebra $H$, $S^{M,N}\circ S_{M,N}$ is the
identity map. On the other hand, $(S_{M,N}\circ S^{M,N})(m  \ot_{H_t} n)=
1_{(1)}.m \ot_{H_t} 1_{(2)}.n$. By \equref{1.5.1}, (rc) implies $1_{(1)}\ot
1_{(2)}\in H\ot H_t$ hence  
$$
1_{(1)}.m \ot_{H_t} 1_{(2)}.n=
\ol\varepsilon_s(1_{(2)})1_{(1)}.m \ot_{H_t} n=
m \ot_{H_t} n,
$$
where the second equality follows by $\ol\varepsilon_s(1_{(2)})1_{(1)}=\lan
\varepsilon,1_{(2')}1_{(2)}\ran 1_{(1')}1_{(1)}=1$. Similarly, by
\equref{1.5.1}, (lc) implies $1_{(1)}\ot 1_{(2)}\in \ol H_s\ot H$ hence 
$$
1_{(1)}.m \ot_{H_t} 1_{(2)}.n=
m \ot_{H_t} \varepsilon_t(1_{(1)})1_{(2)}.n=
m \ot_{H_t} n.
$$
This proves that whenever (rc) or (lc) holds, $S^{M,N}$ and $S_{M,N}$ are mutual
inverses. \\
By \leref{lem:eps_1} and by the second equality in \equref{eq:tbar_s},
\begin{eqnarray*}
&&\hspace{-2cm}1_{(1')}1_{(1)}\ot \varepsilon_t(1_{(2')})\ot 1_{(2)}=
\varepsilon_s(1_{(1')})1_{(1)}\ot 1_{(2')}\ot 1_{(2)}\\
&=&1_{(1)}\ot 1_{(2')}\ot \ol\varepsilon_t(1_{(1')})1_{(2)}=
1_{(1)}\ot\varepsilon_t(1_{(1')})\ot 1_{(2')}1_{(2)}.
\end{eqnarray*}
This means that $1_{(1')}1_{(1)}.m\ot \varepsilon_t(1_{(2')})\ot
1_{(2)}.n=1_{(1)}.m\ot\varepsilon_t(1_{(1')})\ot 1_{(2')}1_{(2)}.n$ for all
$m\in M$ and $n\in N$, so that $M\ot^l N$ is a subspace of the $H_t$-comodule
tensor product $M\ot^{H_t} N$.\\
The opmonoidal functor $U:{}_H\Mm\to {}_k\Mm$ factorizes through an opmonoidal
functor ${}_H\Mm\to {}^{H_t}\Mm^{H_t}$. The binary part of the opmonoidal
structure of the latter functor comes out as the evident inclusion map 
$
M\ot^l N \to M\ot^{H_t} N
$.
Thus in order to prove essentially strong opmonoidality of $U$, we need to
show that $M\ot^{H_t} N$ and $M\ot^l N$ are coinciding subspaces of $M\ot N$.
Since $\sum m^i \ot n^i\in M\ot^{H_t} N$, $\sum 1_{(1)}.m^i \ot
\varepsilon_t(1_{(2)}).n^i=\sum m^i \ot \varepsilon_t(1_{(1)})1_2.n^i=\sum m^i
\ot n^i$; and $\sum \ol\varepsilon_s(1_{(1)}).m^i \ot 1_{(2)}.n^i=\sum
\ol\varepsilon_s(1_{(2)})1_1.m^i \ot  n^i=\sum m^i \ot n^i$. This shows that,
in light of \equref{1.5.1}, any of (rc) and (lc) implies that $M\ot^{H_t} N$
and $M\ot^l N$ coincide, so that $U$ is essentially strong opmonoidal. 
\end{proof}
A particular consequence of this theorem is that, for a monoidal
prebialgebra $H$ for which any of (rc) and (lc) holds, and for any left
$H$-modules $M$ and $N$, there are isomorphisms (of $H$-modules)
$M\ot_l N\cong M\ot^l N\cong M\ot_{H_t} N\cong M\ot^{H_t} N$.\\
It also follows by \thref{thm:forgetful} that a monoidal prebialgebra $H$ for
which any of (rc) and (lc) holds, is a left bialgebroid over $H_t$ (and in
fact also a right bialgebroid over $H_s$) -- some of the axioms occurred in 
\equref{eq:sbar_t}.

\section{The corepresentation category}\selabel{3}
In \seref{2}, we have seen that the category of modules over a monoidal prebialgebra
is monoidal.
Throughout this section, we assume that $H$ is a comonoidal prebialgebra , that is,
it satisfies condition (c). 
Then $I_s=H_s=H_R=\ol{I}_s=\ol{H}_s$, by \leref{1.2}. For $y\in  H_s$, we have
that $\Delta(y)=1_{(1)}\ot  y  1_{(2)} =1_{(1)}\ot 1_{(2)}y\in
   H_s\ot H$, by \equref{1.5.1}, hence $ H_s$ is a right $H$-comodule.\\
Suppose that we have a coassociative $H$-coaction on $M$, which is not necessarily
counital; then we also have an associative $H^*$-action on $M$, given by
$h^*\cdot m=\lan \varepsilon, m_{[1]}\ran m_{[0]}$. If $M,N\in \Mm^H$, then
$M\ot N$ has a coassociative $H$-coaction given by
$$\rho(m\ot n)=m_{[0]}\ot n_{[0]}\ot m_{[1]}n_{[1]}.$$
Now define
$M\ot^r N=\varepsilon\cdot (M\ot N)$, which is the submodule of $M\ot N$ generated by
elements of the form $\lan \varepsilon,m_{[1]}n_{[1]}\ran m_{[0]}\ot n_{[0]}$. It is then
easy to show that $M\ot^r N$ is a counital right $H$-subcomodule of $M\ot N$.
Then $\ot^r$ is an associative tensor product on $\Mm^H$, and the associativity
constraint is trivially induced by the associativity on ${}_k\Mm$. Now consider the
maps
\begin{eqnarray*}
l_M:\  H_s\ot M\to M&;&l_M(y\ot m)=g'(y)\cdot m= \lan \varepsilon, ym_{[1]}\ran m_{[0]};\\
\ol{l}_M:\ M\to  H_s\ot^r M&;& \ol{l}_M(m)=f'(\varepsilon_{(2)})\ot \varepsilon_{(1)}\cdot m=
\ol{\varepsilon}_s(m_{[1]})\ot m_{[0]};\\
r_M:\ M\ot  H_s\to M&;& r_M(m\ot y)= g(y)\cdot m= \lan \varepsilon, m_{[1]}y\ran m_{[0]};\\
\ol{r}_M:\ M\to M\ot^r  H_s&;& \ol{r}_M(m)=\varepsilon_{(2)}\cdot m\ot f'(\varepsilon_{(1)})
=m_{[0]}\ot \varepsilon_s(m_{[1]}).
\end{eqnarray*}
The description of $\ol{l}$ and $\ol{r}$ in terms of $\varepsilon_{(1)}$ and $\varepsilon_{(2)}$
only makes sense if $H^*$ is locally projective; the other description holds in general.

\begin{theorem}\thlabel{3.1}
Let $H$ be a comonoidal prebialgebra.
Then we have a monoidal category $(\Mm^H,\ot^r, H_s,a,l,r)$.
\end{theorem}

\begin{proof}
We first show that $\ol{l}_M$ is right $H$-colinear. For all $m\in M$, we have
\begin{eqnarray*}
&&\hspace*{-15mm}
\rho(\ol{l}_M(m))=\rho(\ol{\varepsilon}_s(m_{[1]})\ot m_{[0]})
=1_{(1)}\ot m_{[0]}\ot 1_{(2)}\ol{\varepsilon}_s(m_{[2]})m_{[1]}\\
&=&1_{(1)}\ot m_{[0]}\ot 1_{(2)}1_{(1')}m_{[1]}\lan \varepsilon,1_{(2')}m_{[2]}\ran
=1_{(1)}\ot m_{[0]}\ot  1_{(2)}m_{[1]}\\
&\equal{\equref{1.1.8}}&
 \ol{\varepsilon}_s(m_{[1]})\ot m_{[0]}\ot m_{[2]}=\ol{l}(m_{[0]})\ot m_{[1]}.
 \end{eqnarray*}
 We now show that the restriction of $l_M$ to $ H_s\ot^r M$ is the inverse of
 $\ol{l}_M$. It is easy to see that $(l_M\circ \ol{l}_M)(m)=m$, for all $m\in M$.
 $ H_s\ot^r M$ is generated by elements of the form $\lan\varepsilon,y_{(2)}m_{[1]}\ran
 y_{(1)}\ot m_{[0]}$. Now
 \begin{eqnarray*}
 &&\hspace*{-1cm}
 l_M(\lan\varepsilon,y_{(2)}m_{[1]}\ran y_{(1)}\ot m_{[0]})
=\lan \varepsilon, y_{(1)}m_{[1]}\ran\lan \varepsilon, y_{(2)}m_{[2]}\ran m_{[0]}
= \lan \varepsilon, y m_{[1]}\ran m_{[0]}\\
&&\hspace*{-1cm}
\ol{l}_M(\lan \varepsilon, y m_{[1]}\ran m_{[0]})
= \lan \varepsilon, y m_{[2]}\ran \ol{\varepsilon}_s(m_{[1]})\ot m_{[0]}\\
&=& \lan \varepsilon, y m_{[2]}\ran  \lan \varepsilon, 1_{(2)}m_{[1]}\ran 1_{(1)} \ot m_{[0]}\\
&=& \lan \varepsilon, y 1_{(2')}m_{[2]}\ran  \lan \varepsilon, 1_{(2)}1_{(1')}m_{[1]}\ran 1_{(1)} \ot m_{[0]}\\
&\equal{\rm (lc)}&
\lan \varepsilon, y 1_{(3)}m_{[2]}\ran  \lan \varepsilon, 1_{(2)}m_{[1]}\ran 1_{(1)} \ot m_{[0]}
= \lan \varepsilon, y 1_{(2)}m_{[1]}\ran 1_{(1)} \ot m_{[0]}\\
&=& 
 \lan\varepsilon,y_{(2)}m_{[1]}\ran y_{(1)}\ot m_{[0]}.
\end{eqnarray*}
Let us now show that the triangle axiom holds.
\begin{eqnarray*}
&&\hspace*{-2cm}
(M\ot \ol{l}_N)(\lan \varepsilon, m_{[1]} n_{[1]}\ran m_{[0]}\ot n_{[0]})
=\lan \varepsilon, m_{[1]} n_{[2]}\ran m_{[0]}\ot \ol{\varepsilon}_s(n_{[1]})\ot n_{[0]}\\
&\equal{\equref{1.1.8}}&
\lan \varepsilon, m_{[1]} 1_{(2)} n_{[1]}\ran m_{[0]}\ot 1_{(1)}\ot n_{[0]}\\
&\equal{\equref{1.1.6}}&
\lan \varepsilon, m_{[2]} n_{[1]}\ran m_{[0]}\ot \varepsilon_s(m_{[1]})\ot n_{[0]}\\
&=& (\ol{r}_M\ot N)(\lan \varepsilon, m_{[1]} n_{[1]}\ran m_{[0]}\ot n_{[0]}).
\end{eqnarray*}
\end{proof}

For $M,N\in \Mm^H$, let $M\ot_r N=(M\ot N)/\Ker({\rm id}-\pi)$, where $\pi(m\ot n)=
\varepsilon\cdot (m\ot n)$. $\pi$ induces a natural isomorphism between $\ot_r$ and
$\ot^r$, the arguments are similar to the arguments following \thref{2.4}.

\begin{theorem}\thlabel{thm:forgetful_co}
Let $H$ be a comonoidal prebialgebra. Then the forgetful functor $U:\mathcal M^H
\to {}_k\Mm$ is monoidal and opmonoidal. If any of (rm) and (lm) holds,
then it is essentially strong monoidal and  essentially strong opmonoidal.
\end{theorem}

\begin{proof}
The monoidal structure of $U$ is given by 
$$
U^{M,N}:\xymatrix{M\ot^r N\ \ar@{>->}[r]&M\ot N};\qquad
U^0:\xymatrix{H_s\ar[r]^-{\lan\varepsilon,-\ran}&k}
$$
and the opmonoidal structure is given by 
$$
U_{M,N}:\xymatrix{M\ot N\ar[r]^-\pi&M\ot^r N};\qquad 
U_0:\xymatrix{k\ar[r]^-{(-)1}&H_s},
$$
for any right $H$-comodules $M$ and $N$. Note that $U_{M,N}\circ
U^{M,N}$ is the identity map. These monoidal and opmonoidal structures
make $U$ a separable Frobenius monoidal functor if and only if $H$ is a weak
bialgebra.\\ 
The monoidal structure of $U$ induces $H_s$-bimodule actions on any right $H$-comodule $M$
$$
y\triangleright m=\lan\varepsilon,ym_{[1]}\ran m_{[0]}\qquad\text{and}\qquad
m\triangleleft y= \lan\varepsilon,m_{[1]}y\ran m_{[0]}
$$
and the opmonoidal structure of $U$ induces $H_s$-bicomodule coactions
$$
m\mapsto m_{[0]}\ot\varepsilon_s(m_{[1]})\qquad\text{and}\qquad
m\mapsto\ol\varepsilon_s(m_{[1]})\ot m_{[0]}.
$$
The monoidal functor $U:\Mm^H\to {}_k\Mm$ factorizes through a monoidal functor
$S:{}_H\Mm\to {}_{H_s}\Mm_{H_s}$. The binary part of the monoidal structure of
$S$ comes out as  
$$
S^{M,N}:M\ot_{H_s} N\to M\ot^r N,\qquad m\ot_{H_s} n \mapsto\lan\varepsilon,m_{[1]} n_{[1]}\ran m_{[0]}\ot n_{[0]}.
$$
(In order to check directly that it is well-defined, note that -- since
$I_s=H_s=\ol H_s=\ol I_s$ -- the $H$-coactions are $H_s$-linear in the sense that  
$$
(m\triangleleft y)_{[0]}\ot (m\triangleleft y)_{[1]}=m_{[0]}\ot m_{[1]}y
\quad\text{and}\quad 
(y\triangleright m)_{[0]}\ot (y\triangleright m)_{[1]}=m_{[0]}\ot ym_{[1]}.
$$
Then the projection $\pi:M\ot N \to M\ot^r N$ is $H_s$-balanced.)
In order to prove essentially strong monoidality of $U$, we need to construct
the inverse of $S^{M,N}$. As a candidate, consider the restriction  
$S_{M,N}:\ M\ot^r N \to M\ot_{H_s} N$
of the canonical
epimorphism $M\ot N \to M\ot_{H_s} N$,
$$
S_{M,N}\bigl(\lan\varepsilon,m_{[1]} n_{[1]}\ran
m_{[0]}\ot n_{[0]}\\bigr)= \lan\varepsilon,m_{[1]} n_{[1]}\ran m_{[0]}\ot_{H_s} n_{[0]}.
$$
The composite $S^{M,N}\circ S_{M,N}$ is the identity map on $M\ot^r N$ and
$S_{M,N}\circ S^{M,N}$ is equal to the idempotent endomorphism of
$M\ot_{H_s} N$ given by
$$
S_{M,N}\circ S^{M,N}(m\ot_{H_s} n )=\lan\varepsilon,m_{[1]} n_{[1]}\ran m_{[0]}\ot_{H_s} n_{[0]}.
$$
If (rm) holds then 
\begin{eqnarray*}
&&\hspace{-2cm}\lan\varepsilon,m_{[1]} n_{[1]}\ran m_{[0]}\ot_{H_s}n_{[0]}\equal{\equref{1.1.2}}
\lan\varepsilon,\varepsilon_s(m_{[1]}) n_{[1]}\ran m_{[0]}\ot_{H_s} n_{[0]}\\
&=&m_{[0]}\ot_{H_s}\varepsilon_s(m_{[1]})\triangleright n=
m_{[0]}\triangleleft \varepsilon_s(m_{[1]})\ot_{H_s} n\\
&=&m_{[0]}\lan\varepsilon,m_{[1]}\varepsilon_s(m_{[2]})\ran\ot_{H_s} n=
m\ot_{H_s} n.
\end{eqnarray*}
If (lm) holds then the same equality follows by \equref{1.1.4}, proving that
$U$ is essentially strong monoidal whenever any of (lm) and (rm) 
holds.\\ 
For any comonoidal prebialgebra $H$, 
\begin{eqnarray*}
&&\hspace{-2cm}\lan\varepsilon,m_{[2]} n_{[1]}\ran m_{[0]}\ot\varepsilon_s(m_{[1]}) \ot n_{[0]}\equal{\equref{1.1.6}}
\lan\varepsilon,m_{[1]}1_{(2)}n_{[1]}\ran m_{[0]}\ot 1_{(1)}\ot n_{[0]}\\
&\equal{\equref{1.1.8}}&\lan\varepsilon,m_{[1]} n_{[2]}\ran m_{[0]}\ot \ol\varepsilon_s(n_{[1]}) \ot n_{[0]}.
\end{eqnarray*}
Thus it follows that $M\ot^r N$ is a subspace of the comodule tensor product
$M\ot^{H_s} N$. \\
The opmonoidal functor $U:\Mm^H\to {}_k\Mm$ factorizes through a unique opmonoidal
functor $\Mm^H\to {}^{H_s}\Mm^{H_s}$. The binary part of the opmonoidal
structure of the latter functor comes out as as the evident inclusion map
$M\ot^r N\to M\ot^{H_s} N$. 
Now if (rm) holds then for any $\sum m^i\ot n^i\in M\ot^{H_s}N$,
\begin{eqnarray*}
&&\hspace{-2cm}\sum\lan\varepsilon,m^i_{[1]} n^i_{[1]}\ran m^i_{[0]}\ot n^i_{[0]}\equal{\equref{1.1.2}}
\sum\lan\varepsilon,\varepsilon_s(m^i_{[1]}) n^i_{[1]}\ran m^i_{[0]}\ot n^i_{[0]}\\
&=&\sum\lan\varepsilon,\ol\varepsilon_s(n^i_{[2]}) n^i_{[1]}\ran m^i\ot n^i_{[0]}=
\sum m^i\ot n^i,
\end{eqnarray*}
so that $M\ot^{H_s} N$ and $M\ot^r N$ are coinciding subspaces of $M\ot N$.
If (lm) holds then the same equality follows by \equref{1.1.4}, proving that
$U$ is essentially strong opmonoidal whenever any of (lm) and (rm) holds.
\end{proof}

As a particular consequence of \thref{thm:forgetful_co}, we have the following
result: given two right comodules $M$ and $N$ over a comonoidal bialgebra
$H$ satisfying (rm) or (lm), we have $H$-colinear isomorphisms
$M\ot_r N\cong M\ot^r N\cong M\ot_{H_s} N\cong M\ot^{H_s} N$.\\
It also follows by \thref{thm:forgetful_co} that a comonoidal prebialgebra $H$
satisfying (rm) or (lm) is a right bicoalgebroid over $H_s$.

\section{Frobenius separable algebras}\selabel{Frob}
In this section, we first collect some properties of Frobenius separable algebras. They are
well-known, and can be found in the literature in various levels of generality. We give
the results with a sketch of proof. 
Recall that a $k$-algebra is called Frobenius if there exists a Frobenius system
$(e,\varepsilon)$, consisting of $e=e^{<1>}\ot e^{<2>}\in A\ot A$ and $\varepsilon\in A^*$
such that $ae=ea$ for all $a\in A$, and $\varepsilon(e^{<1>}) e^{<2>}=
e^{<1>}\varepsilon( e^{<2>})=1$. $A$ is called Frobenius separable if, moreover,
$e^{<1>}e^{<2>}=1$.\\
We have a dual notion for coalgebras. A Frobenius system for a coalgebra $C$
consists of a couple $(\theta,1)$, with $\theta\in (C\ot C)^*$ and $1\in C$
satisfying
\begin{equation}\eqlabel{coalg2.1}
\theta(c\ot d_{(1)})d_{(2)}= c_{(1)} \theta(c_{(2)}\ot d),
\end{equation}
and $\theta(1\ot c)=\theta(c\ot 1)=\varepsilon(c)$, for all $c\in C$. $C$ is
termed coseparable Frobenius if, moreover, $\theta\circ \Delta=\varepsilon$.

\begin{proposition}\prlabel{Frob1}{\bf \cite{Brz}, \cite[Sec. 5.12]{Brz2}, \cite[Cor. 2.6]{CIM},
\cite[Theorem 1.6]{Street}}
Let $A$ be a $k$-module. There is a one-to-one correspondence between
(separable) Frobenius algebra structures on $A$ and (coseparable)
Frobenius coalgebra structures on $A$.
\end{proposition}

\begin{proof}
Let $A$ be an algebra with Frobenius system $(e,\varepsilon)$. Then we define
a comultiplication $\Delta$ on $A$ by $\Delta(a)=ae=ea$. The counit on $A$
will be $\varepsilon$. $\theta\in (A\ot A)^*$ is defined by $\theta(a\ot b)=
\varepsilon(ab)$. Straightforward computations show that $A$ is then a Frobenius
coalgebra with Frobenius system $(\theta,1)$.\\
Conversely, if $A$ is a Frobenius coalgebra with Frobenius system $(\theta,1)$,
then we define a multiplication on $A$ by the formula
$$ab=\theta(a\ot b_{(1)})b_{(2)}= a_{(1)}  \theta(a_{(2)}\ot b).$$
$1$ is a unit for this multiplication, and this makes $A$ into an associative algebra
with Frobenius system $(\Delta(1),\varepsilon)$.\\
It is clear that these two constructions are inverse to each other.
Finally observe that
$e^{<1>}e^{<2>}=1$ is equivalent to $\theta\circ \Delta=\varepsilon$.
\end{proof}

Let $A$ be a coalgebra; for $M\in \Mm^A$ and $N\in {}^A\Mm$, we can consider
the cotensor product
$$M\ot^A N=\{\sum_i m_i\ot n_i\in M\ot N~|~\sum_i m_{i[0]}\ot m_{i[1]}\ot n_i=
\sum_i m_i\ot n_{i[-1]}\ot n_{i[0]}.$$

\begin{proposition}\prlabel{Frob2}
Let $A$ be a Frobenius algebra, and consider the corresponding Frobenius coalgebra
structure on $A$, see \prref{Frob1}. Then we have isomorphisms of categories
$\Mm_A\cong \Mm^A$ and ${}_A\Mm\cong{}^A\Mm$.
\end{proposition}

\begin{proof}
On $M\in \Mm^A$, we define an $A$-action by $ma=m_{[0]}\varepsilon(m_{[1]}a)$.
On $M\in\Mm_A$, we define an $A$-coaction by $\rho(m)=me^{<1>}\ot e^{<2>}$.
\end{proof}

Let $A$ be a separable Frobenius algebra, and consider $M\in \Mm_A\cong \Mm^A$
and $N\in {}_A\Mm\cong{}^A\Mm$. We have a projection
$\pi:\ M\ot N\to M\ot N$, $\pi(m\ot n)=me^{<1>}\ot e^{<2>}n$. It is clear that
$\im(\pi)=\Ker({\rm id}-\pi)$ and $\Ker(\pi)=\im({\rm id}-\pi)$

\begin{proposition}\prlabel{Frob3}{\bf \cite[Remarks 5.2 and 7.5, Prop. 5.3]{Schau},
\cite[p. 164-165]{Pastro}}
With notation as above, we have
$$\im(\pi)=M\ot^A N\cong M\ot_A N=(M\ot N)/\Ker(\pi).$$
\end{proposition}

\begin{proof}
Let $x=\sum_i m_i\ot n_i\in M\ot N$. If $x\in M\ot^A N$, then
$$\sum_i m_ie^{<1>}\ot e^{<2>}\ot n_i=\sum_i m_i\ot e^{<1>}\ot e^{<2>}n_i.$$
Multiplying the second and third tensor factor, and using the separability, we find
that $x=\pi(x)$, so $x\in \im(\pi)$. Conversely, if $x\in \im(\pi)$, then
\begin{eqnarray*}
&&\hspace*{-1cm}
\sum_i m_i \ot e^{<1>}\ot e^{<2>} n_i=
\sum_i m_ie^{<1'>} \ot e^{<1>}\ot e^{<2>} e^{<2'>}n_i\\
&=&\sum_i m_ie^{<1'>} \ot e^{<2'>}e^{<1>}\ot e^{<2>} n_i\\
&=& \sum_i m_ie^{<1>}e^{<1'>} \ot e^{<2'>}\ot e^{<2>} n_i
= \sum_i m_ie^{<1'>}\ot e^{<2'>}\ot n_i,
\end{eqnarray*}
and it follows that $x\in M\ot^A N$.\\
Now let $I$ be the submodule of $M\ot N$ generated by elements of the form
$ma\ot n-m\ot an$. Then $M\ot_A N=(M\ot N)/I$, so it suffices to show that
$I=\Ker(\pi)$. If $x=ma\ot n-m\ot an$, then $\pi(x)=
mae^{<1>}\ot e^{<2>} n-me^{<1>}\ot e^{<2>}an=0$, so $I\subset \Ker(\pi)$.
Take $y= m\ot n-me^{<1>}\ot e^{<2>}n\in \Ker(\pi)=\im({\rm id}-\pi)$.
Then $y= m\ot e^{<1>}e^{<2>}n-me^{<1>}\ot e^{<2>} n\in I$.\\
The canonical surjection $M\ot N\to (M\ot N)/\Ker(\pi)$ restricts to a map
$$p:\ \im(\pi)\to (M\ot N)/\Ker(\pi).$$
$\pi:\ M\ot N\to M\ot N$ induces a well-defined map
$$q:\ (M\ot N)/\Ker(\pi)\to \im (\pi).$$
It is straightforward to show that $q$ is the inverse of $p$. This final part of the
proof uses arguments that are similar to the ones presented  after
Theorems \ref{th:2.4} and \ref{th:3.1}.
\end{proof}

From now on, let $H$ be a weak bialgebra. We have seen in Lemmas \ref{le:1.2} and
\ref{le:1.4}  that $H_t=\ol{H}_t$ and $H_s=\ol{H}_s$ are subalgebras of $H$.
It was shown in \cite[Prop. 1.6]{Sz} that $H_s$ and $H_t$ are in fact
separable Frobenius algebras.

\begin{proposition}\prlabel{alg1}
Let $H$ be a weak bialgebra. Then $H_t$ and $H_s$ are Frobenius separable
$k$-algebras. The separability idempotents of $H_t$ and $H_s$ are
$$e_t=\varepsilon_t(1_{(1)})\ot 1_{(2)}=1_{(2)}\ot \ol{\varepsilon}_t(1_{(1)})~~{\rm and}~~
e_s=1_{(1)}\ot \varepsilon_s(1_{(2)})=\ol{\varepsilon}_s(1_{(2)})\ot 1_{(1)}.$$
The Frobenius systems for $H_t$ and $H_s$ are respectively
$(e_t,\varepsilon_{|H_t})$ and $(e_s, \varepsilon_{|H_s})$.
\end{proposition}

\begin{proof}
We include a proof for the sake of completeness.
Let $e_t=\varepsilon_t(1_{(1)})\ot 1_{(2)}=e^{<1>}\ot e^{<2>}$.
For all $z\in H_t$, we have that
\begin{eqnarray}
&&\hspace*{-2cm}
\lan \varepsilon,z\varepsilon_t(1_{(1)})\ran 1_{(2)}=
\lan \varepsilon, z1_{(2')}\ran \lan\varepsilon, 1_{(1')}1_{(1)}\ran 1_{(2)}\nonumber\\
&\equal{\rm (rm)}&\lan \varepsilon, z1_{(1)}\ran1_{(2)}=\ol{\varepsilon}_t(z)=z;
\eqlabel{alg1.1}\\
&&\hspace*{-2cm}
\varepsilon_t(1_{(1)})\lan \varepsilon,1_{(2)}z\ran
=\lan \varepsilon, 1_{(1')}1_{(1)}\ran \lan \varepsilon, 1_{(2)}z\ran1_{(2')}\nonumber\\
&\equal{\rm (lm)}&\lan \varepsilon, 1_{(1')}z\ran 1_{(2')}=\varepsilon_t(z)=z.
\eqlabel{alg1.2}
\end{eqnarray}
Taking $z=1$ in (\ref{eq:alg1.1}-\ref{eq:alg1.2}), we find that $\varepsilon(e^{<1>})e^{<2>}=
e^{<1>}\varepsilon(e^{<2>})=1$. Then we compute for all $z\in H_t$ that
\begin{eqnarray*}
&&\hspace*{-2cm}
ze^{<1>}\ot e^{<2>}=z\varepsilon_t(1_{(1)})\ot 1_{(2)}
\equal{(*)}\varepsilon_t(1_{(1')})\lan\varepsilon, 1_{(2')}z\varepsilon_t(1_{(1)})\ran\ot 1_{(2)}\\
&\equal{(**)}& \varepsilon_t(1_{(1')})\ot 1_{(2')}z= e^{<1>}\ot e^{<2>}z.
\end{eqnarray*}
At $(*)$, we applied \equref{alg1.2} with $z$ replaced by $z\varepsilon_t(1_{(1)})$,
and at $(**)$, we applied \equref{alg1.1} with $z$ replaced by $1_{(2')}z$.
Finally, it is obvious that $e^{<1>} e^{<2>}=\varepsilon_t(1_{(1)}) 1_{(2)}=1$.
\end{proof}

It follows from Propositions \ref{pr:Frob1} and \ref{pr:alg1} that $H_s$ and $H_t$
are coseparable Frobenius coalgebras. The comultiplication is given by
\begin{eqnarray*}
\Delta_t(z)&=& 
z\varepsilon_t(1_{(1)})\ot 1_{(2)}=
 \varepsilon_t(z1_{(1)})\ot 1_{(2)}
\equal{(\ref{eq:1.10.1}),(\ref{eq:1.5.1})}\varepsilon_t(1_{(1)}z)\ot \varepsilon_t(1_{(2)})
\end{eqnarray*}
(where the second equality follows by \leref{1.4}),  so that the
comultiplication coincides with the comultiplication introduced in \prref{1.12}.\\
For $M,N\in {}_H\Mm$,  in terms of the $H_t$-actions obtained in the
proof of \thref{thm:forgetful},  we compute the projection $\pi:\ M\ot
N\to M\ot N$ from \prref{Frob3}: 
$$
\pi(m\ot n)=m \triangleleft  \varepsilon_t(1_{(1)})\ot 1_{(2)}  \triangleright  n=
(\ol{\varepsilon}_s\circ \varepsilon_t)(1_{(1)})m\ot 1_{(2)}n=1\cdot (m\ot n).$$
Thus $\pi$ coincides with the map $\pi$ introduced  after \thref{2.4}.   
From \prref{Frob3}, we now have  an alternative derivation of  the
following result,  which follows also by \thref{thm:forgetful}: 

\begin{corollary}\colabel{alg4}
Let $H$ be a weak bialgebra, and $M,N\in {}_H\Mm$. Then $M$ and $N$ are 
$H_t$-bimodules and $H_t$-bicomodules, and
$$M  \ot^{H_t}
N=M\ot^l N\cong M\ot_lN=M\ot_{H_t} N.$$
\end{corollary}

For $M,N\in \Mm^H$, we compute  in terms of the $H_s$-actions obtained in
the proof of \thref{thm:forgetful_co}  the projection $\pi:\ M\ot N\to
M\ot N$ from \prref{Frob3}; for $m\in M$, $n\in N$, we find
\begin{eqnarray*}
&&\hspace*{-2cm}
\pi(m\ot n)=m \triangleleft  1_{(1)}\ot \varepsilon_s(1_{(2)}) \triangleright n\\
&=& \lan \varepsilon,m_{[1]}1_{(1)}\ran \lan \varepsilon, 1_{(1')}n_{[1]}\ran
\lan \varepsilon, 1_{(2)}1_{(2')}\ran m_{[0]}\ot n_{[0]}\\
&=& \lan \varepsilon, m_{[1]}n_{[1]}\ran m_{[0]}\ot n_{[0]}
=\varepsilon\cdot (m\ot n).
\end{eqnarray*}
We conclude that $\pi$ coincides with the map $\pi$ introduced  after
\thref{3.1}. 
From \prref{Frob3}, we now have  an alternative derivation of  the
following result,  which follows also by \thref{thm:forgetful_co}: 

\begin{corollary}\colabel{alg5}
Let $H$ be a weak bialgebra, and $M,N\in\Mm^H$. Then $M$ and $N$ are 
$H_s$-bimodules and $H_s$-bicomodules, and
$$M \ot^{H_s}
N=M \ot^r N\cong M\ot_rN=M\ot_{H_s} N.$$
\end{corollary}

\end{document}